\documentclass[final]{siamonline171218}

\usepackage{mathrsfs}
\usepackage{subfig}
\usepackage{lipsum,ulem}
\usepackage{amsfonts,amssymb,verbatim}
\usepackage{graphicx}
\usepackage{epstopdf}
\usepackage{xcolor}
\usepackage{algorithm,algpseudocode}
\usepackage{mathtools}

\usepackage{geometry}
\usepackage{cite}
\usepackage[shadow,prependcaption]{todonotes}
\usepackage{dsfont}
\usepackage{bm}
\ifpdf
  \DeclareGraphicsExtensions{.eps,.pdf,.png,.jpg}
\else
  \DeclareGraphicsExtensions{.eps}
\fi

\graphicspath{{./images/}}

\usepackage[shortlabels]{enumitem}
\setlist[enumerate]{leftmargin=.5in}
\setlist[itemize]{leftmargin=.5in}
\usepackage{amsopn}

\newsiamremark{remark}{Remark}
\newsiamremark{hypothesis}{Hypothesis}
\crefname{hypothesis}{Hypothesis}{Hypotheses}
\newsiamthm{claim}{Claim}
\newsiamthm{example}{Example} %
\newsiamthm{assumptions}{Assumptions} %

\newcommand{\bq}{\begin{equation}}
\newcommand{\eq}{\end{equation}}
\newcommand{\Laplace}{\Delta}

\newcommand{\cL}{\mathcal{L}}
\newcommand{\cC}{\mathcal{C}}
\newcommand{\dee}{\mathrm{d}}
\newcommand{\m}{\lambda} 

\newcommand{\R}{\mathbb R}
\newcommand{\E}{\mathbb E}

\newcommand{\eps}{\varepsilon}
\newcommand{\Q}{\mathbb Q}
\newcommand{\dH}{d_{\mathrm{H}}}
\newcommand{\la}{\langle}
\newcommand{\ra}{\rangle}
\newcommand{\G}{\mathcal{G}}
\def \P {{\mathbb{P}}}
\DeclareMathOperator*{\argmin}{arg\,min}

\numberwithin{theorem}{section}

\newcommand{\TheTitle}{Stability of Gibbs Posteriors from the Wasserstein Loss for Bayesian Full Waveform Inversion} 
\newcommand{\TheAuthors}{M. M. Dunlop, Y. Yang}

\headers{Stability of Gibbs Posteriors for Bayesian FWI}{\TheAuthors}

\title{{\TheTitle}}

\author{
  Matthew M. Dunlop\thanks{Courant Institute of Mathematical Sciences, New York University, New York, New York, 10012, USA (\email{matt.dunlop@nyu.edu, yunan.yang@nyu.edu}).}
  \and
  Yunan Yang\footnotemark[1].
}

\ifpdf
\hypersetup{
  pdftitle={\TheTitle},
  pdfauthor={\TheAuthors}
}
\fi

\begin{document}

\maketitle

\begin{abstract}
Recently, the Wasserstein loss function has been proven to be effective when applied to deterministic full-waveform inversion (FWI) problems. We consider the application of this loss function in Bayesian FWI so that the uncertainty can be captured in the solution. Other loss functions that are commonly used in practice are also considered for comparison. Existence and stability of the resulting Gibbs posteriors are shown on function space under weak assumptions on the prior and model. In particular, the distribution arising from the Wasserstein loss is shown to be quite stable with respect to high-frequency noise in the data. We then illustrate the difference between the resulting distributions numerically, using Laplace approximations to estimate the unknown velocity field and uncertainty associated with the estimates.
\end{abstract}



\tableofcontents

\section{Introduction} \label{sec:intro}
Seismic signals generated by natural earthquakes or induced seismicity contain essential information about subsurface properties. Nowadays, vibrations of the earth could be measured on the surface more accurately and more frequently in spatial and time domain, rather than merely the travel time~\cite{virieux2009overview}. The full wavefield could be generated by partial differential equations (PDE) from the acoustic wave equation to 3D elastic wave modeling with attenuation. The state-of-art imaging technique in geophysics is the full-waveform inversion (FWI)~\cite{virieux2009overview}, which seeks the optimal parameter by minimizing the objective function that measures the data mismatch between the recorded true data and the simulated waveforms produced by the current prediction. In the deterministic setting, it is PDE-constrained optimization. 

Most current FWI studies have focused on how to efficiently and accurately solve the data-fitting problem. However, the recorded data often contain various types of noise that affect the accuracy of inversion~\cite{tarantola2005inverse}. For example, common sources of noise include the surface upon which the survey was performed, the instruments of receiving and recording, and the noise generated by the induced seismicity~\cite{oropeza2011simultaneous}. As a result, the resolution analysis and uncertainty quantification of the predictions are essential aspects~\cite{gouveia1998bayesian,fichtner2011resolution}. Bayesian inference provides a systematic way to quantify uncertainties for geophysical inverse problems~\cite{zhu2016bayesian,dashti2016bayesian}. A realistic noise model is an essential a priori for Bayesian inversion of geophysical data, and the likelihood function is also partially determined by choice of the noise model. However, quantification of the noise model is nontrivial, and the common additive Gaussian assumption might not be enough to characterize the real uncertainty~\cite{motamed2019wasserstein}. 

The shape and curvature of the likelihood surface represent information about the stability of the estimates, whose analogy in the deterministic approach of solving FWI is the objective function in PDE-constrained optimization. The oscillatory and periodic nature of waveforms leads to the main challenge of local minima, which can be significantly mitigated by a recently introduced class of objective functions from optimal transport~\cite{EFWass,engquist2016optimal,W1_2D,yang2017application}. Attractive properties such as the convexity and insensitivity to noise have been theoretically studied in~\cite{yangletter, yang2019analysis,ERY2019}. Since optimal transport studies probability measures, the use of the Wasserstein metrics as criteria is natural in statistical inference~\cite{bernton2017inference,sommerfeld2018inference}, particularly in the Bayesian setting~\cite{el2012bayesian}. In~\cite{motamed2019wasserstein}, the quadratic Wasserstein metric was first used in Bayesian seismic inversion as the likelihood function. This work, together with the success of optimal-transport based metrics in deterministic FWI, motivates us to further analyze the Wasserstein metric under the framework of Bayesian seismic inversion~\cite{zhu2016bayesian,izzatullah2019bayesian}. Compared to~\cite{motamed2019wasserstein}, there are three new contributions to this field from our paper. First, we theoretically study the underlying noise model assumption on the data by choosing the quadratic Wasserstein metric as the likelihood function. Second, we prove the existence and stability of the corresponding posterior distributions, which are formally solutions to Bayesian inverse problems. The benefits of choosing the Wasserstein likelihood stand out from the stability estimates because of its robustness to the high-frequency noise. The noise robustness traces back to both the asymptotic connection~\cite{otto2000generalization} and the non-asymptotic connection~\cite{peyre2018comparison} between the quadratic Wasserstein metric and the negative Sobolev seminorm. The conclusion in this paper is consistent with the analysis of using the quadratic Wasserstein metric in general deterministic inverse data matching problems~\cite{ERY2019}. Third, the numerical examples illustrated here consider the nonlinear inversion of the velocity parameter, which is discretized into thousands of variables. The number of unknowns here is much larger than the ones in~~\cite{motamed2019wasserstein}.

The rest of the paper is organized as follows. We first briefly review necessary background knowledge on optimal transport, the quadratic Wasserstein metric, and the related negative Sobolev norm. As the main application of the paper, seismic waveform inversion is briefly introduced, and the major challenges are addressed. The main contribution of this paper is presented in \cref{sec:Bayesian,sec:proofs}, where Bayesian FWI is studied systematically. Starting with basics in Bayesian inversion, we then outline choices of prior distribution appropriate for seismic inversion~\cite{scales2001prior,asnaashari2013regularized}. Next, we consider different likelihoods/potentials and the corresponding data models. We define four potentials ($\Phi_{\cdot}:X\times Y\to\R$) and discuss the corresponding noise models for each one of them:
\begin{align}
\Phi_{L^2}(u;y) &= \frac{1}{2}\int_D \|\G(u)(x,\cdot)-y(x,\cdot)\|_{L^2(T)}^2\,\lambda(\dee x),\label{eq:L2_potential}\\
\Phi_{\dot H^{-1}}(u;y) &= \frac{1}{2}\int_D\|\G(u)(x,\cdot)-y(x,\cdot)\|_{\dot H^{-1}(T)}^2\,\lambda(\dee x),\label{eq:Hm1_potential}\\
\Phi_{M}(u;y) &= \frac{1}{2}\int_D\left\|\frac{(P_\sigma \G(u))(x,\cdot)-(P_\sigma{y})(x,\cdot)}{(P_\sigma y)(x,\cdot)}\right\|_{L^2(T)}^2\,\lambda(\dee x),\label{eq:M_potential}\\
\Phi_{W_2}(u;y) &= \frac{1}{2}\int_D W_2\left((P_\sigma y)(x,\cdot),(P_\sigma\G(u))(x,\cdot)\right)^2\,\lambda(\dee x),\label{eq:W2_potential}
\end{align}
where the choice of norm on $Y$ depends on the choice of potential; here $P_\sigma$ is an operator that maps functions into probability densities.
In the rest of \cref{sec:Bayesian}, we will rigorously define the posterior distributions \cref{eq:bayes_formal}, and study the existence and stability with respect to perturbations to the observed data for the different choices of likelihood. In \cref{thm:existence}, we establish that the posterior measures corresponding to the likelihoods discussed above are well-defined, and in \cref{thm:stability}, we prove that these measures are stable with respect to perturbations of the observed data measured in different norms. The theoretical analysis demonstrates the advantages of choosing Wasserstein-type likelihood. Numerical simulations are shown in \cref{sec:numerical} to demonstrate the main findings of our study. We will present a challenging benchmark in FWI. We can tackle them using proper choices of likelihood and priors by combining the mathematical tool and the physics knowledge of the geophysical problem. Concluding remarks are offered in \cref{sec:conclusion}.

\section{Background} \label{sec:background}
In this section, we present some background material regarding the quadratic Wasserstein metric, the negative Sobolev norm, optimal transport, and the waveform inversion problem. To better motivate our study, we will also briefly mention the deterministic approach that is reformulated as a wave-equation-constrained optimization.

\subsection{The Quadratic Wasserstein Distance and the Negative Sobolev Norm} \label{ssec:W2}
The Wasserstein distance comes from optimal transport, which is a classical subject in mathematical analysis that was first brought up by Monge in 1781~\cite{Monge} and later expanded by Kantorovich in 1940s~\cite{Kantorovich2006}. The core of the subject, i.e., the optimal transport problem, discusses the optimal plan that maps one probability distribution $\nu_1$ on a measure space $X$ into another probability distribution $\nu_2$ on a measure space $Y$, intending to minimize the total transport cost of a given cost function. The transport cost function $c(x,y)$ maps pairs $(x,y) \in X\times Y$ to $\mathbb{R}\cup \{+\infty\}$, which denotes the cost of transporting one unit mass from location $x$ to $y$. If $c(x,y) = |x-y|^p$ for $p \geq 1$, the  optimal transport cost becomes the class of Wasserstein distance:
\begin{definition}[The Wasserstein distance]
  We denote by $\mathscr{P}_p(X)$ the set of probability measures with finite moments of order $p$ on the measure space $X$. For all $p \in [1, \infty)$ and any $\nu_1, \nu_2\in \mathscr{P}_p(X)$, the $p$-Wasserstein distance between $\nu_1$ and $\nu_2$ is defined as
\begin{equation}\label{eq:static}
W_p(\nu_1,\nu_2)=\left( \inf _{T\in \mathcal{M}_{\nu_1,\nu_2}}\int_{X}\left|x-T(x)\right|^p\, \nu_1(\dee x)\right) ^{\frac{1}{p}}
\end{equation}
where $\mathcal{M}_{\nu_1,\nu_2} = \{T:X\to X\text{ measurable}\,|\,T^\sharp \nu_1 = \nu_2\}$\footnote{Here $T^\sharp\nu_1$ denotes the pushforward of $\nu_1$ by the map $T$, i.e. the measure such that $(T^\sharp \nu_1)(A) = \nu_1(T^{-1}(A))$ for all measurable $A\subseteq X$.} is the set of all measure-preserving maps that rearrange the distribution $\nu_1$ into $\nu_2$.
\end{definition}
In this paper, we are interested in studying the case of $p=2$. With a bit of abuse of notation, we also write $W_2(\nu_1,\nu_2) = W_2(f,g)$ where $f,g$ are density functions of the measures $\nu_1$ and $\nu_2$ that are absolutely continuous with respect to the Lebesgue measure: $\dee\nu_1 = f(x)\dee x$, $\dee\nu_2 = g(x)\dee x$.
The quadratic Wasserstein distance ($W_2$) has close connections with the negative Sobolev space $\dot{H}^{-1}$~\cite{otto2000generalization,dolbeault2009new,ERY2019}. We will present the most relevant results regarding this paper that clearly illustrate the close connections. In~\cref{sec:proofs}, we will analyze them both as the choice of the likelihood function in Bayesian inversion.

We first introduce the \textit{weighted} $L^2$, $\dot{H}^1$ and $\dot{H}^{-1}$ norms on a connected set $T\subseteq \R^s$. Given strictly positive probability density $f = d\nu_1$, we can define a Laplace-type linear operator $L:\dot{H}^1\to\dot{H}^{-1}$
\begin{align}
\label{eq:weighted_lap}
Lh = -\Delta h +\nabla(-\log f)\cdot \nabla h
= -\frac{1}{f}\nabla\cdot(f\nabla h)
\end{align}
which satisfies the fundamental integration by parts formula:
\[
\int_{T} (Lh_1)h_2\,\dee\nu_1 = \int_{T} h_1(Lh_2)\,\dee\nu_1 = \int_{T} \nabla h_1 \cdot \nabla h_2\,\dee\nu_1,
\]
provided $h_1,h_2$ have homogeneous Dirichlet or Neumann boundary conditions.
Therefore, we can define the weighted Sobolev norms $\|h\|_{L^2(f)}$, $\|h\|_{\dot{H}^1(f)}$ and $\|h\|_{\dot{H}^{-1}(f)}$ for any $h$ that satisfies $\int_T h\,\dee\nu_1 = 0$.

\[
\|h\|^2_{L^2(f)} = \int_{T} h^2\,\dee\nu_1 , \quad \|h\|^2_{\dot{H}^1(f)}  = \int_{T}  |\nabla h|^2\,\dee\nu_1, 
\]
\[ ~\label{eq:H-1}
\|h\|^2_{\dot{H}^{-1}(f)} \coloneqq \sup \bigg\{\int_{T} h\varphi\,\dee\nu_1 \  \bigg |\ \|\varphi\|_{\dot{H}^1(f)} \leq 1  \bigg\}^2  = \int_{T} h(L^{-1}h)\,\dee\nu_1.
 = \int_{T} h(L^{-1}h)f\,\dee x.\]

\Cref{thm:local} indicates that the linearization of $W_2$ is \textit{weighted} $\dot{H}^{-1}$, i.e., $W_2(f,g)\approx \|f-g \|_{\dot{H}^{-1}(f)} $ if $g$ is an infinitesimal perturbation of $f$. 
\begin{theorem}[Linearization of $W_2$~\cite{otto2000generalization,Villani}]\label{thm:local}
For any positive probability density functions $f$ and $g_{\varepsilon}=(1+\varepsilon h)f $ on $\mathbb{R}^d$ where $\int_T h d\nu_1= 0$, then
\[
\|h \|_{\dot{H}^{-1}(f)}^2 = \liminf_{\varepsilon\rightarrow0} \frac{W_2(f,g_{\varepsilon})^2}{\varepsilon^2}.
\]
\end{theorem}

\begin{remark}
The $W_2$ distance and the $\dot{H}^{-1}(f)$ norm are quite different metrics if $f$ and $g$ are not close enough in the above linearization regime. For example, $\dot{H}^{-1}(f)$ does not have the global convexity that $W_2$ has with respect to signal translation and dilation~\cite{ERY2019}.
\end{remark}
    
If \cref{thm:local} states the asymptotic behavior of the $W_2$ distance (between two probability distributions that are close enough), the next theory gives a relatively global characterization between the  negative Sobolev norm $\dot{H}^{-1}(T)$ and the $W_2$ distance. We remark that here notation $\dot{H}^{-1}(T)$ represents the familiar $\dot{H}^{-1}$ seminorm for Lebesgue measure, i.e., $f=1$ in Equation~\eqref{eq:H-1}, instead of the weighted $\dot{H}^{-1}$ norm. The theorem provides a \textit{non-asymptotic} comparison in the sense that $f$ and $g$ are not necessary to be close enough in a linearization regime and it still holds for $f$ and $g$ in any dimension.

\begin{theorem}[The Equivalence between $\dot{H}^{-1}$ and $W_2$~\cite{peyre2018comparison}]\label{thm:bound}
Given any positive probability density functions $f$ and $g$ on $T$ that are bounded from below and from above by constants $a,b$ where $0<a<b<+\infty$, i.e., $a<f,g<b$, then
\[
\frac{1}{\sqrt{b}}\|f-g \|_{\dot{H}^{-1}(T)} \leq  W_2(f,g) \leq \frac{1}{\sqrt{a}} \|f-g \|_{\dot{H}^{-1}(T)}
\]
\end{theorem}

The two theorems above shed light on the local behavior of $W_2$ by relating it with the well-known Sobolev norms. From \cref{thm:bound}, we know that if $W_2(f,g)$ is small, then the corresponding negative Sobolev space seminorm  $\dot{H}^{-1}(T)$ is also small. While on the other hand, \cref{thm:local} conveys the message that once $W_2$ drives $f$ close enough to $g$, the distance becomes exactly the \textit{weighted} $\dot{H}^{-1}$ norm.

\subsection{Nonlinear Seismic Inversion in the Deterministic Setting}
\label{ssec:fwd}
For a forward problem, the relation between the input model parameter $u$, and the forward operator $\G$, and the observable output data $y$ can be expressed as
\bq \label{eq:fwd_simple}
y = \G(u).
\eq
The explicit form of $\G$ can be found in~\cref{sec:fwi}.
The inverse problem of reconstructing  $u$ from $y$ typically does not fulfill Hadamard's postulates of well-posedness: there might not be a solution in the strict sense, solutions might not be unique, or a solution might not depend continuously on the data. 
Even if $\G$ is linear, there is no guarantee that it is invertible. Even if $\G$ is invertible, the computational cost of inverting a dense matrix is prohibitive. 
An alternative way of formulating the inverse problem is to estimate the true model parameter $u^*$ through the solution of an optimization problem
\bq \label{eq:simple_misfit}
u^* = \argmin \limits_u J(\G(u), y) + R(u),
\eq
where $J$ is a suitable choice of objective/loss/misfit function characterizing the difference between the data $\G(u)$  generated by the current (and inaccurate) model parameter $u$ and the true observable data $y$. $R(u)$ here denotes the regularization term to enforce desirable properties on the solution. For example, the Tikhonov regularization $R(u) = ||\Gamma u||^2_2$ with a chosen matrix $\Gamma$ is a common choice to obtain a smooth reconstruction of $u$.

The deterministic approach~\eqref{eq:simple_misfit}  is called full waveform inversion (FWI) in exploration geophysics, which is a PDE-constrained optimization by wave equations. The forward operator $\G$ is highly nonlinear for FWI. The conventional objective function $J$ is the least-squares norm ($L^2$)~\cite{virieux2009overview}. 
The frequency content of the data~\cite{ Mora1989} constrains the accuracy of the model parameter in inversion, and high-frequency data is advantageous to achieve higher resolution in the model recovery. However, the oscillatory and periodic nature of waveforms leads to a primary challenge and results in finding only local minima when using the $L^2$ norm. The inversion result is also extremely sensitive to high-frequency noise in the data. Numerous works have been done on the subject to deal with these issues. One particular idea is to replace the $L^2$ norm with other objective functions in optimization for a wider basin of attraction and better stability. Since its first proposal~\cite{EFWass}, the $W_2$ distance~\eqref{eq:static} from optimal transport theory has been extensively studied for topics including the convexity, noise robustness, signal normalization and fast algorithms for computation in the geophysics community\cite{engquist2016optimal,W1_2D,yang2019analysis}. As a new class of objective function for FWI, optimal transport-based metrics are shown to be effective in dealing with local minima issues and are already used in the industry for realistic inversion.

\section{The Bayesian Approach} \label{sec:Bayesian}
In this section, we outline the Bayesian approach to inversion. One of the motivations for using a Bayesian approach is that there is typically uncertainty in the data, for example, from random observational noise, or use of a smoothing/low rank forward model. It then makes sense that uncertainty in the data should be propagated to uncertainty in the solution to the inverse problem. The Bayesian approach combines a probabilistic model for the observed data $y$, $\P(\dee y|u)$, with a probability distribution $\P(\dee u)$ representing our prior belief about the unknown $u$. Bayes' theorem then tells us how to construct the posterior distribution $\P(\dee u|y)$ of the unknown given the data: formally, if $\P(\dee y|u) = \P(y|u)\,\dee y$ admits a Lebesgue density,
\begin{align}
\label{eq:bayes_formal}
\P(\dee u|y) = \frac{\P(y|u)\P(\dee u)}{\P(y)}.
\end{align}
The \textit{probability measure} $\P(\dee u|y)$ is then the solution to the Bayesian inverse problem, rather than a single state as in many classical inversion approaches. This measure can be used to, for example, obtain credible bounds on the solution, or calculate the uncertainty associated with quantities of interest.

Though abstractly the Bayesian approach is quite different from classical approaches, connections between the methods exist. For example, if the prior distribution $\P(\dee u) = N(0,C)$ is chosen to be Gaussian, then modes of the posterior distribution will coincide with minimizers of the classical variational problem
\begin{equation} ~\label{eq:bayes_misfit}
u^* = \argmin_u \Phi(u;y) + \frac{1}{2}\la u,C^{-1} u \ra.
\end{equation}
Here, $\Phi(u,y) = -\log\P(y|u)$ (the negative log-likelihood) represents the misfit between the observed data and the state $u$. The operator $C$ will often be chosen as the inverse of a differential operator, so that the regularization term $\la u,C^{-1} u \ra$ is a Sobolev-type norm, penalizing (lack of) smoothness of the solution. Equivalently, the Tikhonov regularization in~\eqref{eq:simple_misfit} embeds the a priori distribution of $u$. In particular, if the prior distribution $\P(\dee u) \propto \exp(-R(u))$ and the likelihood $\P(y|u) \propto \exp(-J(\G(u),y))$, we have $\Phi(u,y) = J(\G(u),y)$ and the two minimization problems in~\eqref{eq:simple_misfit} and~\eqref{eq:bayes_misfit} coincide.

Alternatively, rather than starting with a probabilistic data model $\P(\dee y|u)$, we can begin with a loss function $\Phi(u;y)$, and define the `likelihood' through the relation $\P(\dee y|u) = \exp(-\Phi(u;y))\,\Q_0(\dee y)$ for some measure $\Q_0$. Note however that we then do not necessarily have that 
\begin{align}
\label{eq:potential_norm}
\P(Y|u) = \int_Y \exp(-\Phi(u;y))\,\Q_0(\dee y) = 1\quad\text{for all}\quad u \in X
\end{align}
and so we do not necessarily have an explicit probabilistic data model. When this is the case, the model is not strictly Bayesian; we will refer to the function $\Phi(u;y)$ as the \textit{potential} and the probability measure $\P(\dee u|y)$ \textit{defined} by \cref{eq:bayes_formal} as the \textit{Gibbs posterior} distribution \cite{alquier2016properties}. This is mainly a philosophical distinction, and the mathematical analysis of the family of measures $\mathbb{P}(\dee u|y)$ will not be affected by whether we work with posterior or Gibbs posterior distributions.

In the next subsections, we will outline choices of prior distribution appropriate for seismic inversion. We will then consider different data models and the corresponding likelihoods. Finally, we will give the details for rigorously defining the (Gibbs) posterior distribution \cref{eq:bayes_formal}, and study its stability with respect to perturbations to the observed data for the different choices of likelihood.

\subsection{The Prior Distribution}
\label{ssec:prior}
In high- and infinite-dimensional Bayesian inverse problems, where the aim is to recover a field, the prior distribution is often chosen to impose properties such as regularity and length-scale on samples. Gaussian priors are often used when continuity of the field is required, and non-Gaussian priors such as Besov and level set priors may be used when this is not desired. In this section, we will outline the details for some choices of prior that will be suitable for seismic inversion.

\subsubsection{Gaussian Priors}
\label{sssec:gaussian_prior}
Gaussian distributions are one of the most studied and utilized classes of distributions on function spaces; we outline their definition, examples and some key properties. Let $D\subseteq \R^d$ denote a spatial domain, and $(\Omega,\mathcal{F},\P)$ a probability space. A random field $u:D\times\Omega\to\R$ is said to be a Gaussian random field if, for any finite collection of points $\{x_j\}_{j=1}^n\subseteq D$, the random vector $(u(x_1,\cdot),\ldots,u(x_n,\cdot))$ is a multivariate Gaussian random variable on $\R^n$. We will often drop the dependence of $u$ on its random argument for cleaner notation.

A useful property of Gaussian random fields is that they are defined completely by their mean function $m:D\to\R$ and covariance function $c:D\times D\to\R$:
\[
m(x) = \mathbb{E}(u(x)),\quad c(x,x') = \mathbb{E}(u(x)-m(x))(u(x')-m(x')),\quad x,x' \in D.
\]
We will write $u \sim \mathsf{GP}(m(x),c(x,x'))$ when this is the case. When $c(x,x')$ depends only on the Euclidean distance $|x-x'|$, the Gaussian field is said to be isotropic. A common family of isotropic covariance functions used in practice are the Mat\'ern covariance functions, defined by
\[
c(x,x') = \sigma^2\frac{2^{1-\nu}}{\Gamma(\nu)}\left(\frac{|x-x'|}{\ell}\right) K_\nu \left(\frac{|x-x'|}{\ell}\right)
\]
for scalar parameters $\sigma,\nu,\ell$ representing amplitude, regularity and length-scale respectively; here $K_\nu$ denotes the modified Bessel function of the second kind. Assuming a smooth mean function $m$, on regular enough domains samples from a Gaussian with the covariance function will almost surely posess up to $\nu$ Sobolev and H\"older derivatives.

It is often convenient to work with the covariance operator $C:L^2(D)\to L^2(D)$, defined by $C = \E (u-m)\otimes (u-m)$ rather than the covariance function, in which case we will write $u \sim N(m,C)$. They are related by
\[
(C\varphi)(x) = \int_D c(x,x')\varphi(x')\,\dee x',\quad \varphi \in L^2(D).
\]
Hence when the covariance function is a Green's function, the covariance operator is the inverse of the corresponding differential operator. When $D = \R^d$, the covariance operator corresponding to the Mat\'ern covariance function above is given by
\[
C = \sigma^2 \frac{\Gamma(\nu+d/2)(4\pi)^{d/2}\ell^d}{\Gamma(\nu)}(I - \ell^2\Delta)^{-\nu-d/2}
\]
where $\Delta$ denotes the Laplace operator on $\R^d$. We can therefore generate samples $u \sim N(0,C)$ by solving the (fractional) SPDE
\[
(I - \ell^2\Delta)^{\nu/2+d/4}u = \sigma\sqrt{\frac{\Gamma(\nu+d/2)(4\pi)^{d/2}\ell^d}{\Gamma(\nu)}} W
\]
where $W$ is Gaussian white noise: $W \sim \mathsf{GP}(0,\delta(x-x'))$, or equivalently $W \sim N(0,I)$.

In \cref{fig:prior_gauss}, the left block depicts examples of samples of Gaussian fields with Mat\'ern covariance functions for fixed $\sigma, \nu$, and $\ell$ decreasing from left-to-right, top-to-bottom.

\subsubsection{Level-Set Type Priors}
\label{sssec:level}
Gaussian random fields, such as those with Mat\'ern covariance as described above, typically have global smoothness properties associated with their samples. In some situations, such as classification problems or inference of salt models discussed later, piecewise constant or piecewise continuous samples may be desired instead. One method to produce such fields is to write them as nonlinear transformations of Gaussian fields. For example, given a Gaussian measure $\nu_0 = N(m,C)$ and scalar values $u_+,u_- \in \R$, one could define a prior measure by the pushforward
\begin{align}
\label{eq:level_set_prior0}
\pi_0 = F^\sharp \nu_0,\quad F(v)(x) = u_+\mathds{1}_{v(x)>0} + u_-\mathds{1}_{v(x)\leq 0}.
\end{align}
That is, $\pi_0$ is the law of the thresholded Gaussian field $F(v), v \sim \nu_0$: samples from $\pi_0$ take the values $u_+,u_-$ almost everywhere, with interface between the values given by the level set $\{v(x) = 0\}$. Such priors have been studied previously from a nonparametric Bayesian perspective \cite{iglesias2016bayesian,DIS16}

Alternatively, one may desire some combination of the above and plain Gaussian priors. For example, given a product Gaussian measure $\nu_0 = N(m_1,C_1)\times N(m_2,C_2)$, one could define a mixed level set prior:
\begin{align}
\label{eq:level_set_prior}
\pi_0 = F^\sharp \nu_0,\quad F(v,w)(x) = u_+\mathds{1}_{v(x)>0} + w(x)\mathds{1}_{v(x)\leq 0}.
\end{align}
Examples of samples from the above two priors are shown in the middle and right blocks of \cref{fig:prior_gauss} respectively, with the length-scale of all underlying Gaussian fields decreasing from left-to-right, top-to-bottom. These examples may be generalized further to multiple interfaces, for example by using the vector level set method \cite{bertozzi2018uncertainty}. Additionally, they may be generalized to allow for uncertainty in the values $u_+$, $u_-$ using a hierarchical method \cite{dunbar2020reconciling}; such priors will be considered numerically in \cref{sec:numerical}.

\begin{remark}
It may be convenient, both theoretically and numerically, to view the prior as $\nu_0$ and compose the potential $\Phi$ with the map $F$. This is sometimes referred to as non-centering, and we will use this when establishing existence and well-posedness of the posterior distributions when using the level-set type priors. It can checked, using the definition of the pushforward measure, that the stability estimates obtained are invariant under this change of parameterization.
\end{remark}

\begin{figure}
\includegraphics[width=\textwidth,trim=5cm 0cm 3.5cm 0cm,clip]{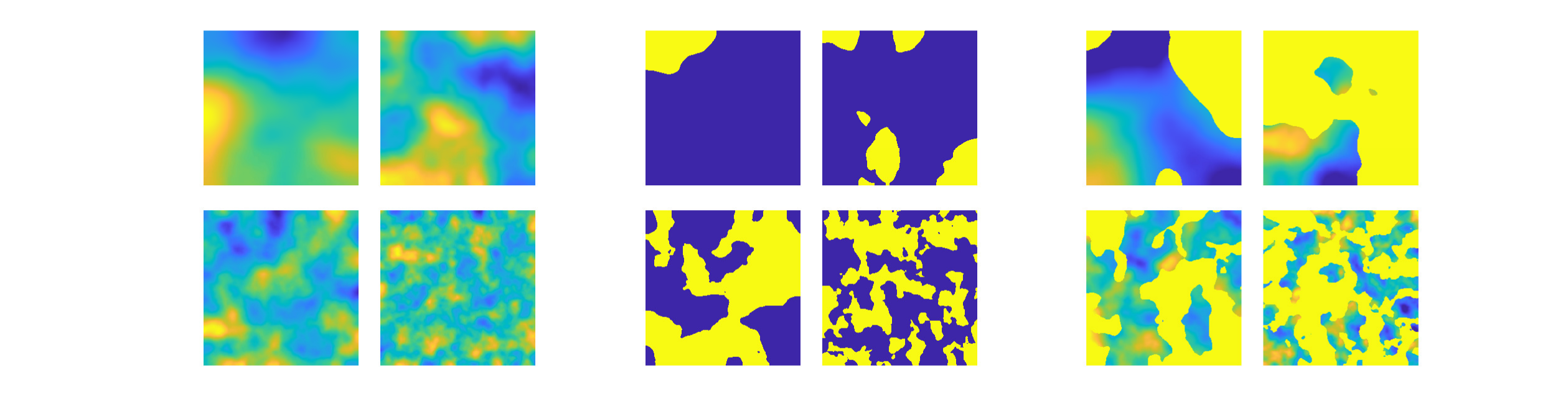}
\caption{Example of independent samples from the priors discussed in \cref{ssec:prior}. (Left) plain Gaussian prior, (middle) plain level set prior and (right) mixed level set prior. In all cases the underlying Gaussian random fields have Mat\'ern covariance with the regularity and amplitude parameters fixed, and the length-scale is decreased from left-to-right, top-to-bottom within each block.}
\label{fig:prior_gauss}
\end{figure}

\subsection{Likelihood and Loss functions}
In this subsection, we outline several likelihood functions and loss functions that are the main focus of this paper. We first consider some natural likelihood functions corresponding to explicit data models, which lead to Bayesian posterior distributions. We then consider two loss functions used in variational inversion approaches, which lead to Gibbs posteriors. In particular, we discuss the noise models that correspond to these posteriors heuristically. In what follows, we assume that the data is a function defined on a spatial domain $D\subseteq\R^d$ equipped with a finite measure $\lambda$, and temporal domain $T\subseteq\R^s$ equipped with the Lebesgue measure. Note that this setup allows us to consider model setups with either discrete or continuous spatial measurements simultaneously --  for example $D$ may be a discrete set equipped with the counting measure, or it may be a continuous set equipped with the Lebesgue or Hausdorff measure.

\subsubsection{Explicit Likelihood Functions}
\label{sssec:explicit}
We first consider the simplest case wherein the loss $\Phi(u;y) = J(\G(u),y)$ is given by the $L^2$ misfit:
\[
\Phi_{L^2}(u;y) = \frac{1}{2}\int_D \|\G(u)(x,\cdot)-y(x,\cdot)\|^2_{L^2(T)}\,\lambda(\dee x).
\]
This loss arises as a negative log-likelihood by assuming that the data is corrupted by additive Gaussian space-time white noise: 
\[
y = \G(u) + \eta,\quad \eta \sim N(0,I).
\]
Note that under this model, if $\dim(Y) = \infty$, the data $y$ will almost-surely not be valued in $L^2$ due to the roughness of the noise: we will have $y(x,\cdot) \in H^{-r}(T)$ almost surely for all $r > s/2$ but not $r = s/2$. To see how this leads to the above loss, observe that $y|u \sim \Q_u:= N(\G(u),I)$. Defining $\Q_0 = N(0,I)$, we have by the Cameron-Martin theorem that
\begin{align*}
\frac{\dee\Q_u}{\dee\Q_0}(y) &= \exp\left(-\frac{1}{2}\|\G(u)\|_{L^2(D;L^2(T))}^2 + \la y,\G(u)\ra_{L^2(D;L^2(T))}\right)\\
&= \exp\left(-\frac{1}{2}\|\G(u)-y\|_{L^2(D;L^2(T))}^2 + \frac{1}{2}\|y\|_{L^2(D;L^2(T))}^2\right)
\end{align*}
since the Cameron-Martin space of white noise is $L^2$. For fixed $y$, if $\dim(Y) < \infty$ or $y \in L^2$ we may drop the second term in the exponent when defining the loss -- its appearance here ensures that the exponent is finite almost-surely under the noise model\footnote{In practice for the FWI problem it will be the case that $\dim(Y) < \infty$, since only a finite number of frequencies will be observed at a finite number of receivers.}.

Another loss function that has been considered is the $\dot H^{-1}$ loss\footnote{Note that $\|y(x,\cdot)\|_{\dot H^{-1}(T)}$ is well-defined only if $\int_T y(x,t)\,\dee t = 0$. After data preprocessing, zero-frequency in time signals are removed from FWI data $y(x,\cdot)$. The mean-zero property makes $\dot H^{-1}$ a proper data discrepancy for comparing seismic data.}
\[
\Phi_{\dot H^{-1}}(u;y) = \frac{1}{2}\int_D \|\G(u)(x,\cdot)-y(x,\cdot)\|^2_{\dot H^{-1}(T)}\,\lambda(\dee x).
\]
This arises from a similar data model above, except instead of assuming that the noise is white, temporal correlations are assumed:
\[
y = \G(u) + \eta,\quad \eta\sim N(0,\Gamma)
\]
where $\Gamma = -\Delta_T$ is the negative Laplacian on the temporal variable. Similarly to the above we have $y|u \sim \Q_u:= N(\G(u),\Gamma)$, and so defining $\Q_0 = N(0,\Gamma)$ we have that
\[
\frac{\dee\Q_u}{\dee\Q_0}(y) = \exp\left(-\frac{1}{2}\|\G(u)-y\|_{\Gamma}^2 + \frac{1}{2}\|y\|_{\Gamma}^2\right)
\]
where $\la\cdot,\cdot\ra_\Gamma := \la\cdot,\Gamma^{-1}\cdot\ra_{L^2(D;L^2(T))} = \la\cdot,\cdot\ra_{L^2(D;\dot H^{-1}(T))}$ denotes the Cameron-Martin inner product assocated with $\Q_0$. Note that in this case the noise will be even rougher in time than the white noise case: since the eigenvalues $\{\lambda_j\}$ of the Laplacian on $T$ asymptotically satisfy $\lambda_j \asymp j^{2/s}$ by a Weyl-type law, we have that if $\eta \sim N(0,\Gamma)$,
\[
\E\|\eta(x,\cdot)\|_{H^{-r}(T)}^2 \asymp \sum_{j=1}^\infty j^{-2r/s+2/s} < \infty\quad\text{iff }r> 1 + \frac{s}{2},
\]
using the Karhunen-Lo\`eve expansion of $\eta(x,\cdot)$. As in the white noise case, we may drop the second term if $\dim(Y)<\infty$ or if $y \in L^2(D;\dot H^{-1}(T))$.

\subsubsection{Loss Functions and Approximate Noise Models}
\label{sssec:implicit_noise}
The loss functions above correspond to negative log-likelihoods from explicit data models, which will lead to Bayesian posteriors if coupled with prior distributions. However, this is not the case for general loss functions, as discussed at the beginning of~\cref{sec:Bayesian}. 
Given a general loss function $\Phi(u;y)$ and a measure $\Q_0$ on $Y$ such that $\exp(-\Phi(u;y))$ is integrable with respect to $\Q_0$, one can define the normalized loss function $\tilde\Phi(u;y)$ by
\[
\tilde\Phi(u;y) = \Phi(u;y) - \log\int_Y \exp\left(-\Phi(u;z)\right)\,\Q_0(\dee z)
\]
so that the relation \cref{eq:potential_norm} holds for $\tilde\Phi(u;y)$. The corresponding data-generating distribution would then be given by $\P(\dee y|u) = \exp(-\tilde\Phi(u;y))\,\Q_0(\dee y)$. However, performing this normalization may be intractable if it cannot be done analytically. Therefore, it may be preferable to work with the unnormalized loss function. In this subsection, we consider an unnormalized multiplicative noise loss function and a Wasserstein loss function. We will show that the latter could be viewed asymptotically as an unnormalized state-dependent multiplicative noise loss in the small noise limit.

In the above additive noise models, the loss function scale was dictated by the size of the noise on the observations. In the absence of an explicit data model, one needs to be careful about choosing the scale for the potential. Typically, one introduces a scalar parameter $\beta > 0$, often referred to as the inverse temperature, and works with $\beta\Phi(u;y)$ in place of $\Phi(u;y)$. The choice of this parameter is known as calibration and may either be chosen empirically \cite{syring2018calibrating} or treated as a hyperparameter as part of the inverse problem \cite{zhu2016bayesian}. We do not discuss the choice of $\beta$ here and assume it to be fixed.

We first introduce the Wasserstein loss function. The relevant background is presented in~\cref{ssec:W2}. To evaluate the quadratic Wasserstein distance ($W_2$) between the reference data and the output from the forward map, we must first transform the datasets into probability densities with respect to the temporal variable. Hence, given a scalar function $\sigma:\R\to\R^+$, we define the normalization operator $P_\sigma$ on functions $y:D\times T\to\R$ by
\[
(P_\sigma y)(x,t) = \frac{1}{Z_\sigma(x)}\sigma(y(x,t)),\quad Z_\sigma(x) = \int_T\sigma(y(x,t'))\,\dee t'.
\]
Given this operator, we then define the Wasserstein loss by
\[
\Phi_{W_2}(u;y) = \frac{1}{2}\int_D W_2\left((P_\sigma\G(u))(x,\cdot),(P_\sigma y)(x,\cdot)\right)^2\,\lambda(\dee x).
\]
Note that $\Phi_{W_2}$ is not normalized in the sense described above and does not appear to correspond to a particular data model even if it were normalized. However, via linearization of the $W_2$ distance, we can describe an approximate data model that this loss corresponds to in the limit of small observational noise. Assume that there is $\eta:D\times T\to\R$ such that 
\[
P_\sigma y = (1+\eta)P_\sigma\G(u),\quad \int_T \eta(x,t)(P_\sigma\G(u))(x,t)\,\dee t = 0\quad\text{for all}\quad x \in D.
\]
Then by~\cref{thm:local}, we may approximate for small $\|\eta(x,\cdot)\|_{\dot H^{-1}(P_\sigma\G(u))}$
\begin{align*}
\Phi_{W_2}(u;y) &\coloneqq \frac{1}{2}\int_D W_2\left((P_\sigma\G(u))(x,\cdot),(P_\sigma y)(x,\cdot)\right)^2\,\lambda(\dee x) \\& = \frac{1}{2}\int_D W_2\bigg((P_\sigma\G(u))(x,\cdot), \big(1+ \eta(x,\cdot)\big)  (P_\sigma \G(u))(x,\cdot )\bigg)^2\,\lambda(\dee x)\\
&\approx \frac{1}{2}\int_D \|\eta(x,\cdot)\|_{\dot H^{-1}(P_\sigma\G(u))}^2\,\lambda(\dee x) \quad \text{(by \cref{thm:local} with $\eps h = \eta$)}\\
&= \frac{1}{2}\int_D\left\|\frac{ (P_\sigma y)(x,\cdot) - P_\sigma\G(u))(x,\cdot) }{(P_\sigma\G(u))(x,\cdot) }\right\|_{\dot H^{-1}(P_\sigma\G(u))}^2\,\lambda(\dee x).
\end{align*} 
The above potential formally corresponds to the negative logarithm of an unnormalized Gaussian density $N(1,\cL(u))$, for some operator $\cL(u)$ defined below, evaluated at the ratio $P_\sigma y/P_\sigma\G(u)$.
It suggests that the Wasserstein loss could formally be considered as asymptotically coming from the state-dependent multiplicative noise data model
\[
P_\sigma y =(1+ \eta)\, P_\sigma\G(u),\quad  \eta|u \sim N(0,\cL(u))
\]
where the operator $\cL(u):\mathrm{Dom}(\cL(u))\to L^2(D;L^2(T))$ is defined by
\begin{align*}
\cL(u)\varphi &= -\frac{1}{P_\sigma \G(u)}\nabla_T\cdot\left(P_\sigma \G(u)\nabla_T \left( \frac{\varphi}{P_\sigma \G(u)} \right)\right),\\
\quad \mathrm{Dom}(\cL(u)) &= \left\{\varphi \in L^2(D;H^2(T))\,\bigg|\,\int_T \varphi(x,t)(P_\sigma\G(u))(x,t)\,\dee t = 0\text{ for all }x \in D\right\}
\end{align*}
and $\nabla_T$ denotes the gradient with respect to the temporal variable; this operator is derived from the definition of the weighted $\dot{H}^{-1}$ norm and the operator \cref{eq:weighted_lap}. Note that subject to appropriate boundedness and regularity of $P_\sigma\G(u)$, samples $\eta \sim N(1,\cL(u))$ will have the same negative Sobolev regularity as samples $\eta \sim N(0,\Gamma)$ considered in \cref{sssec:explicit}.

The relation of Wasserstein loss to a multiplicative noise model is interesting. The latter has been studied in the literature of Bayesian inverse problems~\cite{isaac2015scalable} before the introduction of the Wasserstein loss for FWI~\cite{engquist2016optimal}. One may consider the following unnormalized multiplicative noise model~\cite{iglesias2018bayesian}: 
\[
\Phi_M(u;y) = \frac{1}{2}\int_D \left\|\frac{(P_\sigma\G(u))(x,\cdot)-(P_\sigma y)(x,\cdot)}{(P_\sigma y)(x,\cdot)}\right\|_{L^2(T)}^2\,\lambda(\dee x),
\]
which can be viewed as arising from the model
\[
P_\sigma y = \eta\cdot P_\sigma\G(u),\quad 1/\eta \sim N(1,I).
\]
Alternatively, one may formulate it by informally assuming the noise variance is proportional to the size of the observed data. In general, the data and the output of the forward map do not need to be probability densities in this model. However, for the stability of the resulting posterior, the forward map must be bounded away from zero; see \cite{dunlop2019multiplicative} for a discussion. The condition can be ensured by using the same operator $P_\sigma$, and we do so in the following section for brevity. Additionally, one may prefer to use a model wherein $\eta  \sim N(1, I)$ rather than its reciprocal, in which case $P_\sigma\G(u)$ will replace $P_\sigma y$ in the denominator; we do not provide details of the proofs for this modification, but note that they are similar and slightly simpler than without this modification.

\subsection{The (Gibbs) Posterior Distribution}

In this section, we establish both that the (Gibbs) posterior measures corresponding to the likelihoods/potentials discussed above are well-defined (existence), and that these measures are stable with respect to perturbations of the observed data (well-posedness). The existence is established using the theory/assumptions of \cite{sullivan2017well}, which is a generalization of the result in \cite{dashti2016bayesian} that provides an abstract statement of Bayes' theorem at the level of measures. Given a prior measure $\pi_0$ on a space $X$ and a potential $\Phi$ of the form considered in the previous subsection, we show that the measure $\pi_\Phi^y$ defined by
\[
\pi_\Phi^y(\dee u) = \frac{1}{Z_\Phi(y)}\exp(-\Phi(u;y))\,\pi_0(\dee u),\quad Z_\Phi(y) = \int_X \exp(-\Phi(u;y))\,\pi_0(\dee u),
\]
defines a Radon probability measure on $X$. The definition of $\pi_\Phi^y$ is essentially a restatement of Bayes' theorem as given in \cref{eq:bayes_formal}.

Stability of the posterior will be established with respect to the \textit{Hellinger} distance $d_\mathrm{H}$ on probability measures:
\[
d_\mathrm{H}(\pi,\pi')^2 = \int_X \left(\sqrt{\frac{\dee \pi}{\dee \nu_0}(u)} - \sqrt{\frac{\dee \pi'}{\dee \nu_0}(u)}\right)^2\,\nu_0(\dee u),
\]
where the measure $\nu_0$ is such that both $\pi,\pi'$ are absolutely continuous with respect to $\nu_0$\footnote{This definition can be seen to be independent of the choice of $\nu_0$. Such a measure $\nu_0$ may always be found in practice, for example, $\nu_0 = \frac{1}{2}(\pi + \pi')$, however in our setup we may always take it to be the prior measure $\nu_0 = \pi_0$.}. That is, we show if a perturbation is made to the observed data, then the Hellinger distance between the resulting posteriors is bounded above by some norm of the perturbation. We, in particular, show that for the Wasserstein and $\dot H^{-1}$ choices of likelihood, if the observed data is perturbed, then the Hellinger distance between the corresponding posteriors is bounded above by the $\dot H^{-1}$ norm of this perturbation. This is in contrast to the choice of $L^2$ likelihood, in which it is only bounded by the $L^2$ norm of the perturbation. This essentially shows that the Wasserstein and $\dot H^{-1}$ likelihoods lead to posterior distributions, which are \textit{more robust} with respect to high-frequency noise on the data.

A useful property of the Hellinger distance is that it allows us to bound expectations of quantities of interest:
\[
\|\E^\pi(f) - \E^{\pi'}(f)\|_S \leq C(f) d_{\mathrm{H}}(\pi,\pi')
\]
for any $f \in L^2(X,\pi;S)\cap L^2(X,\pi';S)$. For example, the choice $f(u) = u$ provides us with stability of the posterior means with respect to the data.

\section{Existence and Well-posedness of the (Gibbs) Posterior}
\label{sec:proofs}
In this section, we provide assumptions that lead to the existence and well-posedness of the (Gibbs) posteriors arising from the likelihoods and priors in the previous section. We then show the applicability of the theory to the full-waveform inversion problem under a large class of prior distributions.

\subsection{Notation, Assumptions and Supporting Lemmas}
In what follows $(X,\|\cdot\|_X)$ will denote a separable Banach space. $(D,\mathcal{D},\lambda)$ will denote a finite measure space representing the spatial domain, and $T\subset\R^p$ will denote a compact subset of Euclidean space, representing the temporal domain. Any $L^p$ space over $D$ will be with respect to $\lambda$, and any $L^p$ space over $T$ will be with respect to the Lebesgue measure. We define the following subsets of $L^2(D;L^2(T))$:
\begin{align*}
Y_0 &= \left\{y \in L^\infty(D;L^\infty(T))\,\bigg|\,\int_T y(x,t)\,\dee t = 0\text{ for all }x \in D\right\},\\
Y_{1,\eps} &= \left\{\rho \in L^\infty(D;L^\infty(T))\,\bigg|\,\int_T \rho(x,t)\,\dee t = 1\text{ for all }x \in D,\; \eps \leq \rho \leq \eps^{-1}\right\},\quad \eps > 0.
\end{align*}
Thus, $Y_0$ is a set of bounded functions with no zero-frequency component, and $Y_{1,\eps}$ is a set of probability densities bounded above and below by positive constants. On any set $Y_{1,\eps}$, we have the following equivalence of metrics, which follows directly from  \cref{thm:bound}:

\begin{lemma}
\label{lem:w2h1_equiv}
Let $\eps > 0$, then for any $\rho,\rho' \in Y_{1,\eps}$,
\[
\eps\|\rho-\rho'\|_{L^2(D;\dot H^{-1}(T))}^2 \leq \int_D W_2(\rho(x,\cdot),\rho'(x,\cdot))^2\,\lambda(\dee x) \leq \eps^{-1}\|\rho-\rho'\|_{L^2(D;\dot H^{-1}(T))}^2.
\]
\end{lemma}

We also recall the definition of the normalization operator $P_\sigma$, given a map $\sigma:\R\to\R_+$:
\[
(P_\sigma y)(x,t) = \frac{1}{Z_\sigma(x)}\sigma(y(x,t)),\quad Z_\sigma(x) = \int_T\sigma(y(x,t'))\,\dee t'.
\]
Under appropriate assumptions on the map $\sigma$, we have the following result.
\begin{lemma}
\label{lem:psig}
Let $\sigma:\R\to\R_+$ be locally Lipschitz, and given $r> 0$ let $y,y' \in Y_0$ with $\|y\|_{L^\infty(D;L^\infty(T))},\|y'\|_{L^\infty(D;L^\infty(T))} < r$. Then there exists $\eps(r) > 0$ such that $P_\sigma y,P_\sigma y' \in Y_{1,\eps(r)}$, and $L_\sigma(r)$ such that
\begin{align*}
\|P_\sigma y - P_\sigma y'\|_{L^2(D;L^2(T))} &\leq L_\sigma(r)\|y-y'\|_{L^2(D;L^2(T))},\\
\|P_\sigma y - P_\sigma y'\|_{L^2(D;\dot H^{-1}(T))} &\leq L_\sigma(r)\|y-y'\|_{L^2(D;\dot H^{-1}(T))}.
\end{align*}
\end{lemma}

\begin{proof}
The map $\sigma:\R\to\R_+$ is locally Lipschitz and hence locally bounded. It follows that since $|y|,|y'| < r$ almost everywhere, there exists $k(r) > 0$ such that $k(r) \leq \sigma(y(x,t)),\sigma(y'(x,t)) \leq k(r)^{-1}$ for almost all $(x,t) \in D\times T$, and so $P_\sigma y,P_\sigma y' \in Y_{1,\eps(r)}$ for some $\eps(r) > 0$ by construction of the map $P_\sigma$. 

With $Z = L^2(T)$ or $Z = \dot{H}^{-1}(T)$, we have that for all $x \in D$,
\begin{align*}
\|(P_\sigma y)(x,\cdot) - (P_\sigma y')(x,\cdot))\|_{Z} = \sup\left\{\int_T |(P_\sigma y)(x,t)-(P_\sigma y')(x,t)|\varphi(t)\,\dee t\,\bigg|\,\|\varphi\|_{Z^*} \leq 1\right\}.
\end{align*}
Again using the boundedness $|y|,|y'| < r$, we have $\int_T\sigma(y(x,t))\,\dee t,\int_T\sigma(y'(x,t))\,\dee t \geq |T|k(r) > 0$, and so for all $x \in D$
\begin{align*}
\|(P_\sigma y)(x,\cdot) - (P_\sigma y')(x,\cdot)\|_{Z} &\leq \frac{1}{|T|k(r)}\sup\left\{\int_T |\sigma(y(x,t))-\sigma(y'(x,t))|\varphi(t)\,\dee t\,\bigg|\,\|\varphi\|_{Z^*} \leq 1\right\}\\
&\leq \frac{L(r)}{|T|k(r)}\sup\left\{\int_T |y(x,t)-y'(x,t)|\varphi(t)\,\dee t\,\bigg|\,\|\varphi\|_{Z^*} \leq 1\right\}\\
&= \frac{L(r)}{|T|k(r)}\|y-y'\|_{Z}.
\end{align*}
where $L(r)$ is the local Lipschitz coefficient of $\sigma$. Squaring and integrating over $D$ gives the result.
\end{proof}

In the theory that follows we will make the use of the following assumptions.
\begin{assumptions}
\label{ass:wp}
The data $y$, forward map $\G$ and prior $\pi_0$ satisfy the following:
\begin{enumerate}[(i)]
\item $y \in Y_0$;
\item $\G:X\to Y_0$ is continuous with respect to the $L^2(D;L^2(T))$ norm $\pi_0$-almost surely; and
\item there exists an increasing function $R_\G:\R_+\to\R_+$ such that $\|\G(u)\|_{L^\infty(D;L^\infty(T))} \leq R_\G(\|u\|_X)$ for all $u \in X$.
\end{enumerate}
\end{assumptions}

\subsection{Existence and Well-posedness}

We recall the definitions of the four loss functions $\Phi_{\cdot}:X\times Y\to\R$ introduced in \cref{sec:Bayesian}, and define the norms we equip the corresponding data spaces $Y = Y_0$ with:
\begin{align*}
\Phi_{L^2}(u;y) &= \frac{1}{2}\int_D \|\G(u)(x,\cdot)-y(x,\cdot)\|_{L^2(T)}^2\,\m(\dee x), \quad \|\cdot\|_Y = \|\cdot\|_{L^2(D;L^2(T))},\\
\Phi_{H^{-1}}(u;y) &= \frac{1}{2}\int_D\|\G(u)(x,\cdot)-y(x,\cdot)\|_{\dot H^{-1}(T)}^2\,\m(\dee x), \|\cdot\|_Y = \|\cdot\|_{L^2(D;\dot H^{-1}(T))},\\
\Phi_{M}(u;y) &= \frac{1}{2}\int_D\left\|\frac{(P_\sigma \G(u))(x,\cdot)-(P_\sigma y)(x,\cdot)}{(P_\sigma y)(x,\cdot)}\right\|_{L^2(T)}^2\,\m(\dee x), \quad \|\cdot\|_Y = \|\cdot\|_{L^2(D;L^2(T))},\\
\Phi_{W_2}(u;y) &= \frac{1}{2}\int_D W_2\left((P_\sigma \G(u))(x,\cdot),(P_\sigma y)(x,\cdot)\right)^2\,\m(\dee x), \quad \|\cdot\|_Y = \|\cdot\|_{L^2(D;\dot{H}^{-1}(T))}.
\end{align*}

We may then establish existence and well-posedness of the corresponding (Gibbs) posterior distributions:
\begin{theorem}[Existence] \label{thm:existence}
Let $\pi_0$ be a Borel probability measure on $X$ and let \cref{ass:wp} hold. Then for any choice $\Phi \in \{\Phi_{L^2},\Phi_{H^{-1}},\Phi_M,\Phi_{W_2}\}$,
\[
Z_\Phi(y) = \int_X \exp(-\Phi(u;y))\,\pi_0(\dee u)
\]
is strictly positive and finite, and
\[
\pi_\Phi^y(\dee u) := \frac{1}{Z_\Phi(y)}\exp\left(-\Phi(u;y)\right)\,\pi_0(\dee u)
\]
defines a Radon probability measure on $X$.
\end{theorem}

\begin{proof}
We show that the assumptions of Theorem 4.3 in \cite{sullivan2017well} are satisfied for each choice of $\Phi$. Note that in \cite{sullivan2017well}, the data space is assumed to be a quasi-Banach space, and so in particular complete, however, studying the proofs therein this completeness property is never used -- an incomplete norm structure on the data space suffices to obtain the same results. This is important for our applications, as we equip $Y$ with different incomplete norms. We also note that continuity of the likelihood in the data component is not required for the proof of existence, either by examining the proof in \cite{sullivan2017well} or using Theorem 2.4 in \cite{latz2019well}.

We must establish the following properties for each choice of $\Phi$:
\begin{enumerate}[(i)]
\item $\Phi(\cdot;y)$ is measurable for each $y \in Y$;
\item For each $r>0$ there exists $M_{0,r} \in \R$ such that for each $u\in X$, $y \in Y$ with $\|u\|_X, \|y\|_Y < r$, $|\Phi(u;y)| \leq M_{0,r}$.
\item For each $r>0$ there exists a measurable $M_{1,r}:\R_+\to\R$ such that for each $u \in X$, $y \in Y$ with $\|y\|_Y < r$, $\Phi(u;y) \geq M_{1,r}(\|u\|_X)$, and
\[
\int_X \exp(-M_{1,r}(\|u\|_X))\,\pi_0(\dee u) < \infty.
\]
\end{enumerate}
The latter is trivially satisfied for any probability measure $\pi_0$ by all choices of $\Phi$ considered with $M_{1,r}(\|u\|)\equiv 0$. We verify the other two properties for each $\Phi$ in turn.

\begin{itemize}
\item[$\boxed{\Phi_{L^2}}$]
\begin{enumerate}[(i)]
\item The map $\G$ is assumed to be $\pi_0$-a.s continuous from $X$ into $Y_0$ equipped with the $L^2(D;L^2(T))$ topology, and the $L^2(D;L^2(T))$ norm is continuous with respect to its own topology, so measurability follows.
\item Fix $r > 0$ and choose $u\in X$, $y \in Y$ with $\|u\|_X,\|y\|_Y < r$, then by assumption $\|\G(u)\|_{L^2(D;L^2(T))}^2 \leq \lambda(D)|T|\|\G(u)\|_{L^\infty(D;L^\infty(T))} \leq \lambda(D)|T|R_\G(r)^2$. We may then bound
\begin{align*}
|\Phi_{L^2}(u;y)| &\leq \frac{1}{2}(\|\G(u)\|_{L^2(D;L^2(T))} + \|y\|_{L^2(D;L^2(T))})^2\\
&\leq \frac{1}{2}\m(D)|T|(\|\G(u)\|_{L^\infty(D;L^\infty(T))} + \|y\|_{L^\infty(D;L^\infty(T))})^2\\
&\leq \frac{1}{2}\m(D)|T|(R_\G(r)+r)^2 =: M_{0,r}.
\end{align*}
\end{enumerate}

\item[$\boxed{\Phi_{H^{-1}}}$]
\begin{enumerate}[(i)]
\item  The map $\G$ is assumed to be $\pi_0$-a.s continuous from $X$ into $Y_0$ equipped with the $L^2(D;L^2(T))$ topology, and the embedding of $(Y_0,\|\cdot\|_{L^2(D;\dot L^2(T))})$ into\linebreak $(Y_0,\|\cdot\|_{L^2(D;\dot H^{-1}(T))})$ is continuous. Since the $L^2(D;\dot H^{-1}(T))$ norm is continuous with respect to its own topology, measurability follows.
\item Choose such $u\in X$, $y \in Y$. Then we obtain a similar bound $M_{0,r}$ as for $\Phi_{L^2}$, using that
\[
\|\G(u)\|_{L^2(D;\dot H^{-1}(T))}^2 \leq C_p(T)\|\G(u)\|_{L^2(D;L^2(T))}^2 \leq C_p(T)\lambda(D)|T|R_\G(r)^2,
\]
where $C_p(T)$ is the Poincar\'e constant of the domain $T$.
\end{enumerate}

\item[$\boxed{\Phi_{M}}$]
\begin{enumerate}[(i)]
\item By assumption $\G$ is $\pi_0$-a.s continuous into $Y_0$ equipped with the $L^2(D;L^2(T))$ norm. By \cref{lem:psig} the map $P_\sigma$ is continuous, as is the mapping $\rho\mapsto (\rho-P_\sigma y)/P_\sigma y$ since $P_\sigma y > 0$ for all $y \in Y_0$. Measurability then follows by continuity of the $L^2(D;L^2(D))$ norm. 
\item Choose such $u\in X$, $y \in Y$, then $\|\G(u)\|_{L^\infty(D;L^\infty(T))} \leq R_\G(r)$; without loss of generality assume $R_\G(r) \geq r$. By \cref{lem:psig}, there exists $\eps(r) > 0$ such that $P_\sigma(y) \geq \eps(r)$, and so
\begin{align*}
|\Phi_M(u;y)| &\leq \frac{1}{2}\eps(r)^{-2}\|P_\sigma \G(u)-P_\sigma y\|_{L^2(D;L^2(T))}^2\\
&\leq \frac{1}{2}\eps(r)^{-2}L_\sigma(R_\G(r))^2\|\G(u)-y\|_{L^2(D;L^2(T))}^2.
\end{align*}
The bound now follows from the bound for $\Phi_{L^2}$.
\end{enumerate}

\item[$\boxed{\Phi_{W}}$]
\begin{enumerate}[(i)]
\item This follows from \cref{lem:w2h1_equiv} and the above arguments for $\Phi_{\dot H^{-1}}$.
\item Choose such $u\in X$, $y \in Y$, then $\|\G(u)\|_{L^\infty(D;L^\infty(T))} \leq R_\G(r)$; again without loss of generality assume $R_\G(r) \geq r$. Using \cref{lem:psig}, there exist $\eps(r) > 0, \tilde\eps(r) > 0$ such that $P_\sigma y \in Y_{1,\eps(r)},P_\sigma \G(u) \in Y_{1,\tilde\eps(r)}$. We may therefore use \cref{lem:w2h1_equiv} and \cref{lem:psig} again to bound
\begin{align*}
|\Phi_{W_2}(u;y)| &\leq \frac{1}{2}(\eps(r)\wedge{\tilde\eps(r)})^{-1}\|P_\sigma \G(u)-P_\sigma y\|_{L^2(D;\dot H^{-1}(T))}^2\\
&\leq \frac{1}{2}(\eps(r)\wedge{\tilde\eps(r)})^{-1}L_\sigma(R_\G(r))^2\|\G(u)-y\|_{L^2(D;\dot H^{-1}(T))}^2\\
&\leq \frac{1}{2}(\eps(r)\wedge{\tilde\eps(r)})^{-1}L_\sigma(R_\G(r))^2\m(D)|T|\|\G(u)-y\|_{L^\infty(D;L^\infty(T))}^2\\
&\leq \frac{1}{2}(\eps(r)\wedge{\tilde\eps(r)})^{-1}L_\sigma(R_\G(r))^2\m(D)|T|(R_\G(r)+r)^2 =: M_{0,r},
\end{align*}
where $a\wedge b$ denotes $\min\{a,b\}$.
\end{enumerate}
\end{itemize}
\end{proof}

The following analysis will demonstrate the regularity of the posterior with respect to perturbations of the data. First, we assert the following lemma regarding the local Lipschitz properties of the potentials that are considered here.

\begin{lemma}~\label{lem:lip}
Let \cref{ass:wp}(iii) hold and choose $\Phi \in \{\Phi_{L^2},\Phi_{H^{-1}},\Phi_M,\Phi_{W_2}\}$. For each $r > 0$ there exists $M_{2,r}:\R_+\to\R_+$ such that for each $u \in X$, $y,y' \in Y$ with $\|y\|_{L^\infty(D;L^\infty(T))},\|y'\|_{L^\infty(D;L^\infty(T))} < r$,
\[
|\Phi(u;y)-\Phi(u;y')| \leq M_{2,r}(\|u\|_X)\|y-y'\|_Y.
\]
\end{lemma}

\begin{proof}
We verify for each choice of $\Phi$ in turn; given such a $\Phi$, fix $r > 0$, and choose $u \in X$, $y,y' \in Y$ with $\|y\|_{L^\infty(D;L^\infty(T))},\|y'\|_{L^\infty(D;L^\infty(T))} < r$.
\begin{itemize}
\item[$\boxed{\Phi_{L^2}}$]
We have that $\|\cdot\|_Y = \|\cdot\|_{L^2(D;L^2(T))}$. We may calculate, using the assumption,
\begin{align*}
|\Phi_{L^2}(u;y)-\Phi_{L^2}(u;y')| &= \left|\la \G(u),y'\ra_Y - \la \G(u),y\ra_Y + \frac{1}{2}\|y\|_Y^2 - \frac{1}{2}\|y'\|_Y^2 \right|\\
&\leq |\la \G(u),y-y'\ra_Y| + \frac{1}{2}\left|\|y\|_Y^2-\|y'\|_Y^2\right|\\
&\leq \|\G(u)\|_Y\|y-y'\|_Y + \frac{1}{2}(\|y\|_Y+\|y'\|_Y)\|y-y'\|_Y\\
&\leq \lambda(D)^{1/2}|T|^{1/2}(R_\G(\|u\|_X)+r)\|y-y'\|_Y,
\end{align*}
and so we may take $M_{2,r}(\|u\|_X) = \lambda(D)^{1/2}|T|^{1/2}(R_\G(\|u\|_X)+r)$.

\item[$\boxed{\Phi_{H^{-1}}}$]
We have that $\|\cdot\|_Y = \|\cdot\|_{L^2(D;\dot H^{-1}(T))}$. We may calculate similarly to the $L^2$ case above to deduce that we may take
\[
M_{2,r}(\|u\|_X) =C_p(T)^{1/2}\lambda(D)^{1/2}|T|^{1/2}(R_\G(\|u\|_X)+r),
\]
where $C_p(T)$ is the Poincar\'e constant of the domain $T$.

\item[$\boxed{\Phi_{M}}$]
We have that $\|\cdot\|_Y = \|\cdot\|_{L^2(D;L^2(T))}$. Using \cref{lem:psig}, we have that $P_\sigma y,P_\sigma y' \in Y_{1,\eps(r)}$. Using the assumption, we may then calculate, again using \cref{lem:psig},
\begin{align*}
|\Phi_M(u;y)&-\Phi_M(u;y')| = \frac{1}{2}\left|\left\|\frac{P_\sigma \G(u)-P_\sigma y}{P_\sigma y}\right\|_Y^2 - \left\|\frac{P_\sigma \G(u)-P_\sigma y'}{P_\sigma y'}\right\|_Y^2\right|\\
&= \frac{1}{2}\left|\left\|\frac{P_\sigma \G(u)}{P_\sigma y}\right\|_Y^2 - \left\|\frac{P_\sigma \G(u)}{P_\sigma y'}\right\|_Y^2 + 2\left\langle\frac{P_\sigma\G(u)}{P_\sigma y'}-\frac{P_\sigma \G(u)}{P_\sigma y},1\right\rangle_Y\right|\\
&\leq \left(\left\|\frac{P_\sigma \G(u)}{P_\sigma y}\right\|_Y+\left\|\frac{P_\sigma \G(u)}{P_\sigma y'}\right\|_Y + \|1\|_Y\right)\left\|\frac{P_\sigma \G(u)}{P_\sigma y}-\frac{P_\sigma \G(u)}{P_\sigma y' }\right\|_Y\\
&\leq \left(2\eps(r)^{-1}\left\|P_\sigma \G(u)\right\|_Y + \lambda(D)^{1/2}|T|^{1/2}\right)\\
&\hspace{4cm}\times \eps(r)^{-2}\|P_\sigma \G(u) \|_Y\left\|P_\sigma y-P_\sigma y'\right\|_Y\\
&\leq \left(2\eps(r)^{-1}\eps(R_\G(\|u\|_X))^{-1} + \lambda(D)^{1/2}|T|^{1/2}\right)\\
&\hspace{4cm}\times \eps(r)^{-2}\eps(R_\G(\|u\|_X))^{-1}L_\sigma(r)\left\|y-y'\right\|_Y,
\end{align*}
which provides $M_{2,r}(\|u\|_X)$.

\item[$\boxed{\Phi_{W_2}}$]
We have that $\|\cdot\|_Y = \|\cdot\|_{L^2(D;\dot H^{-1}(T))}$. By \cref{lem:psig}, there exist $\eps(r),\tilde\eps(\|u\|_X) > 0$ such that $P_\sigma y,P_\sigma y' \in Y_{1,\eps(r)}$, $P_\sigma \G(u) \in Y_{1,\tilde\eps(\|u\|_X)}$. Using the (reverse) triangle inequality for $W_2(\cdot,\cdot)$, Cauchy-Schwarz, \cref{lem:w2h1_equiv,lem:psig} and the assumption, we may calculate
\begin{equation*}
\begin{split}
|&\Phi_{W_2}(u;y)-\Phi_{W_2}(u;y')|\\
&\leq \frac{1}{2}\int_D \big|W_2((P_\sigma\G(u))(x,\cdot),(P_\sigma y)(x,\cdot))^2 - W_2((P_\sigma \G(u))(x,\cdot),(P_\sigma y')(x,\cdot))^2\big|\,\m(\dee x)\\
&\leq \frac{1}{2}\int_D \left|W_2((P_\sigma\G(u))(x,\cdot),(P_\sigma y)(x,\cdot)) + W_2((P_\sigma \G(u))(x,\cdot),(P_\sigma y')(x,\cdot))\right|\\
&\hspace{7.6cm}\times W_2((P_\sigma y)(x,\cdot),(P_\sigma y')(x,\cdot))\,\m(\dee x)\\
&\leq (\tilde\eps(\|u\|_X)\wedge\eps(r))^{-1/2}\eps(r)^{-1/2}(\|P_\sigma\G(u)-P_\sigma y\|_Y + \|P_\sigma \G(u)-P_\sigma y'\|_Y)\\
&\hspace{10.7cm}\times\|P_\sigma y-P_\sigma y'\|_Y\\
&\leq 2(\tilde\eps(\|u\|_X)\wedge\eps(r))^{-1/2}\eps(r)^{-1/2}C_p(T)^{1/2}\lambda(D)^{1/2}\\
&\hspace{6.9cm}\times(\tilde\eps(\|u\|_X)^{-1} + \eps(r)^{-1})L_\sigma(r)\|y-y'\|_Y,
\end{split}
\end{equation*}
which provides $M_{2,r}(\|u\|_X)$.
\end{itemize}
\end{proof}

The following now follows from \cref{lem:lip} and Theorem 4.4 in \cite{sullivan2017well}.
\begin{theorem}[Well-posedness] \label{thm:stability}
Let $\pi_0$ be a Borel probability measure on $X$ and let \cref{ass:wp} hold. Choose $\Phi \in \{\Phi_{L^2},\Phi_{H^{-1}},\Phi_M,\Phi_{W_2}\}$, let $r > 0$ and assume
\[
S_{\Phi,r} := \int_X M_{2,r}(\|u\|_X)^2\,\pi_0(\dee u) < \infty,
\]
where $M_{2,r}(\cdot)$ is the corresponding function from \cref{lem:lip}. Then there exists $C_\Phi(r) > 0$ such that for all $y,y' \in Y$ with $\|y\|_{L^\infty(D;L^\infty(T))},\|y'\|_{L^\infty(D;L^\infty(T))} < r$,
\begin{align*}
\dH(\pi_{\Phi}^y,\pi_{\Phi}^{y'}) \leq C_\Phi(r)\|y-y'\|_{Y}.
\end{align*}
\end{theorem}

\begin{remark}

We address that $\|\cdot\|_Y = \|\cdot\|_{L^2(D;\dot H^{-1}(T))}$ for $\Phi_{H^{-1}}$ and $\Phi_{W_2}$ while $\|\cdot\|_Y = \|\cdot\|_{L^2(D;L^2(T))}$ for $\Phi_{L^2}$ and $\Phi_{M}$ from the proof of~\cref{lem:lip}. In particular, in combination with~\cref{thm:stability}, we have the following two stability results for $W_2$-based posterior and $L^2$-based posterior
$$d_{\mathrm{H}}(\pi_{\Phi_{W_2}}^y,\pi_{\Phi_{W_2}}^{y'}) \leq C_{W_2}(r)\|y-y'\|_{L^2(D;\bm{\dot H^{-1}(T)})},$$ $$d_{\mathrm{H}}(\pi_{\Phi_{L^2}}^y,\pi_{\Phi_{L^2}}^{y'}) \leq C_{L^2}(r)\|y-y'\|_{L^2(D;\bm{L^2(T)})},$$
where $d_{\mathrm{H}}$ represents the Hellinger distance. 

The main difference between these two stability estimates lies in the functional spaces for the noise $y-y'$ defined on the domain $T\subseteq\R^s$. As an $L^2$-based Sobolev seminorm, $\dot H^{-1}$ can be considered as a weighted $L^2$ norm with a weighting factor being asymptotically $|\mathbf{\boldsymbol \xi}|^{-1}$ where $\boldsymbol{\xi}$ is the frequency. Thus, the high-frequency content of the input is imposed with a much smaller weight compared to the low-frequency content of the input. In another word, $\|y-y'\|_{L^2(D;\bm{\dot H^{-1}(T)})}$ is much smaller than $\|y-y'\|_{L^2(D;\bm{L^2(T)})}$ given high-frequency noise $y-y'$. These two different estimates consequently lead to a stronger stability result for $W_2$-based posterior and $\dot H^{-1}$-based posterior in the regime of high-frequency noises.

\cref{thm:stability} is another theoretical derivation that demonstrates the better stability of the $W_2$ metric, but in a Bayesian inference framework in the context of solving Bayesian inverse problems. The result is well-aligned with the analysis for the deterministic approach for inverse data matching problems~\cite{ERY2019}. The asymptotic and non-asymptotic connection between the $W_2$ metric and the $\dot H^{-1}$ seminorm play the central role in achieving such estimates as weaker norms smooth out the input.

\end{remark}

\begin{remark}
\label{rem:m2bound}
Key to the applying this result is the assumption that $S_{\Phi,r} < \infty$, which from the definitions of $M_{2,r}(\cdot)$ in \cref{lem:lip} will depend on the growth rate of the forward map $\G$ and the structure of the density mapping $P_\sigma$, as well as the prior $\pi_0$. In the main application considered in this article, FWI, $\G$ will be uniformly bounded and so ensuring this condition holds involves balancing the growth/decay rates of the map $\sigma$ with the tails of the prior. The Gaussian and level-set  priors discussed in \cref{ssec:prior} all have at least squared exponential moments, and so relatively weak conditions are required on $\sigma$ to ensure the required integrability.
\end{remark}

\subsection{Application to Full Waveform Inversion}~\label{sec:fwi}
We show that the above theory is applicable when the forward model is taken to be that described in \cref{ssec:fwd}, with the choices of priors described in \cref{ssec:prior}. Specifically, we define the forward map $\G:X\to Y$ as follows.

Denote $\R^d_U = \R^{d-1}\times\R_+$ the upper half space, and $T\subset (0,\infty)$ the temporal domain with $|T|<\infty$. 
Let $\{s_j\}_{j=1}^{N_s}$ be a collection of point sources, where $s_j(x,t) = w(t)\delta(x-x_j)$, $x_j\in \R^d_U$, and the time-dependent function $w(t) \in H^1(T) \subset L^{\infty}(T)$. Let the space for the state  parameter $X = C^0(\R^d_U;\R^k)$. We then define a transport function $F:X\to L^\infty(\R^d_U;\R_+)$ that maps the state parameter $u$ to the physical parameter $m$ (for Bayesian inversion). We will specify the choice of $F$ later in the numerical examples. Given $u \in X$, we define the slowness $m:\R^d_U\to (v^{-2}_{\max},v^{-2}_{\min}) \subseteq\R_+$ by $m = F(u)$, where the constants $v_{\min}$ and $v_{\max}$ are physical lower and upper bounds for the wave velocity. For each $j$, let $v_j \in L^2(\R^d_U;L^2(T))$ solve the following wave equation
\begin{equation}\label{eq:wave}
\begin{cases}
m(x)\frac{\partial^2 v_j}{\partial t^2}(x,t) - \Laplace v_j(x,t) = s_j(x,t) & (x,t) \in \R^d_U\times T,\\
\hfill v_j(x,0) = 0,\, \frac{\partial v_j}{\partial t}(x, 0 ) = 0   & x \in \R^d_U, \\
\hfill \nabla_x v_j(x,t)\cdot \mathbf{n} = 0 & x\in \partial \R^d_U.\\
\end{cases}
\end{equation}

Let $D_0\subset \R^d_U$ be a compact subset denoting the receiver locations; we assume $D_0$ is either finite or has positive Lebesgue measure, and equip it with either the counting measure or Lebesgue measure respectively. 
.

Given a bounded linear observation operator $\mathcal{O}:L^2(\R^d_U;L^2(T))\to L^2(D_0;L^2(T))$ we write $v_j^D = \mathcal{O}\circ v_j$ as the observable part of the solution.\footnote{The observation operator $\mathcal{O}$ will typically be restriction to the subset $D_0\times T$, though depending on the choice of $D_0$ may require mollification in space to ensure the operator is bounded.} Additionally, we assume the observable part of the solution to be essentially bounded and thus in $L^\infty(D_0;L^\infty(T))$. We concatenate the $v_j^D$ for the $N_s$ different source terms, writing the total amount of observable data $v^D \in Y_0 = L^\infty(D;L^\infty(T))$, where $D = D_0\times\{1,\ldots,N_s\}$. The forward mapping $\G$ is then defined as the mapping $u \mapsto v^D$.

 We first note that though the forward problem is defined on an infinite spatial domain, we are only interested in inversion on a compact subset -- numerically the forward problem will be restricted to this subset using ABC/PML to mimic the true physics \cite{engquist1977absorbing}. The state parameter and slowness function will hence be assumed to only be defined on this compact subset rather than all of $\R^d_U$. We verify that \cref{ass:wp} are satisfied by the above with this assumption in place:
\begin{itemize}
\item Assumption (i) will hold, assuming that the data is preprocessed to remove the zero-frequency component.

\item The continuous dependence of the map $\G$ in the $L^2$ norm with respect to the wave slowness $m(x) = v^{-2}(x)$ was addressed in~\cite[Theorem 2.8.3]{stolk2000modeling}. Assuming the transport map $F:X\to L^\infty$ is continuous almost-surely under the prior $\pi_0$ with respect to the $L^2$ norm, the required $\pi_0$-a.s continuity of the map $\G$ with respect to the state parameter $u$, i.e., Assumption (ii), will follow. If the map $F$ is defined pointwise so that $F(u)(x) = f(u(x))$ for some continuous bounded $f:\R^k\to\R_+$, the required continuity will follow for any choice of prior on $X$. In the case of a level set prior \cref{eq:level_set_prior0}, $F$ is continuous into $L^2$ at $u \in X$ if the set $\{x\,|\,u(x)=0\}$ has zero Lebesgue measure: this can be ensured to hold $\pi_0$-a.s. by choosing the measure $\nu_0$ in \cref{eq:level_set_prior0} to be supported on a space of $C^1$ functions, for example a sufficiently smooth Gaussian as discussed in \cref{sssec:gaussian_prior} \cite{iglesias2016bayesian}. The mixed level set prior \cref{eq:level_set_prior} leads to the required $\pi_0$-a.s. continuity with an analogous setup.

\item Since we have zero initial conditions and the source pulse $w(t)$ has a finite $H^1$ norm (and thus a finite $L^\infty$ norm), we assume that the $L^\infty$ energy estimates for the wave equation~\eqref{eq:wave} guarantee that $\G(u) < C$ where $C$ only depends on $\|w\|_{H^1(T)}, v_{\min}, v_{\max}$, and is independent of the model parameter $u$. As discussed in \cref{rem:m2bound}, since the bound on $\G$ is global, the condition that $S_{\Phi,r} < \infty$ in \cref{thm:stability} is true under weak conditions on the prior measure $\pi_0$ and density mapping $P_\sigma$ -- in particular, the choices of prior outlined in \cref{ssec:prior} and any  $\sigma$ with up to exponential growth/decay.
\end{itemize}

For this particular FWI problem, the temporal domain $T$ in our theoretical framework becomes a 1D domain. Therefore, one can efficiently calculate the four potentials, $\Phi_{L^2},\Phi_{H^{-1}},\Phi_M,\Phi_{W_2}$ as defined in~\eqref{eq:L2_potential}-\eqref{eq:W2_potential}. In particular, the calculation of the $W_2$ potential, $\Phi_{W_2}$, is equivalent to the so-called trace-by-trace approach in~\cite{yang2017application}, for which there is an explicit solution to the corresponding 1D optimal transport problem.

\section{Numerical Examples} \label{sec:numerical}

We illustrate the difference between the resulting distributions corresponding to the different potentials numerically for the FWI problem described in \cref{sec:fwi}. First, we present a simple constant-velocity example to illustrate the different structure among several likelihood functions and their corresponding (Gibbs) posterior distributions using a Gaussian prior. We then consider the recovery of a continuous velocity field using a Gaussian prior. Laplace approximations are first made to the posteriors, which correspond to second-order Taylor approximations of the posteriors around their modes. We look at the effect of perturbation of the data by noise uncorrelated in time on these Laplace approximations; we may compute various metrics between probability measures explicitly in the case of Gaussian. Moreover, we consider a mixed level set prior for the recovery of a salt model, where the velocity within the salt is unknown a priori. We again use Laplace approximations to provide approximations to the posteriors efficiently.

Throughout this section we fix the transformation map $F:X\to L^\infty$ in \cref{sec:fwi} to be given by $F(u)(x) = \alpha_-\tanh(u(x)) + \alpha_+$ and $\alpha_\pm = v_{\min}^{-2} \pm v_{\max}^{-2}$. Hence for any $u(x) \in \R$, the velocity $v(x) = 1/\sqrt{m(x)} = 1/\sqrt{F(u(x))}  \in (v_{\min},v_{\max})$. The hyperbolic tangent function is chosen as a smooth monotone mapping from the unbounded range of Gaussian processes into a positive bounded interval to match the physical meaning of the unknown parameter within the hyperbolic equation.

For the all numerical tests, the wave equation is solved on a rectangular domain and the boundaries are imposed with the absorbing boundary conditions~\cite{engquist1977absorbing} to approximate the unbounded domain and reduce reflections from the upper surface. The acoustic wave equation~\eqref{eq:wave} is solved numerically by a second-order finite difference scheme.

\subsection{Simple illustration of the posterior distributions}
We first illustrate the difference in the potential functions by plotting a slice of the unnormalized likelihood functions and their posteriors in the direction of constant wave speed. The reference velocity is $3$ km/s discretized on a rectangle domain $1.5$ km by $2$ km with a spatial spacing of $10$ m. The source pulse is a superposition of three Ricker wavelets of different phase shifts, centered at $20$ Hz. We equally distribute $8$ point sources at $z=50$ m and $200$ receivers at $z=1.4$ km. The observed data $y$ is recorded for $1.5$ s, containing transmissions only. Since these potentials have different scaling, we rescale them with their values evaluated at $v(x) \approx 1.395$ (i.e., $u(x) = 0$) for illustration purpose. We impose the same Gaussian prior with mean $v(x)\approx 2.63$ ($u(x)=-1$) for all three potentials shown in \cref{fig:simple_posterior}. 

The shape of the unnormalized likelihood functions highlight one significant advantage of the $W_2$ metric for FWI --- mitigating cycle-skipping issues. The unimodal feature in both the $W_2$ likelihood function and its posterior corresponds to the desirable convexity property in the deterministic FWI~\cite{yang2017application} that heavily replies on local optimization algorithms. In turn, it facilitates sampling the posterior for Bayesian inversion. On the other hand, the multimodal feature in the plots for $\dot{H}^{-1}$ and $L^2$ are unsurprising as they compare signals locally, failing to capture the large phase changes in the waveforms. As expected, the shapes improve in the posteriors once the prior information is imposed, regularizing the likelihood function. We remark that the unimodality of $W_2$ is unlikely to hold for general coefficient function $v(x)$, but its multimodality is expected to be much less severe than $\dot{H}^{-1}$ and $L^2$ based on abundance numerical evidence in deterministic FWI~\cite{yang2017application}.

\begin{figure}
\centering
\includegraphics[width=0.45\textwidth]{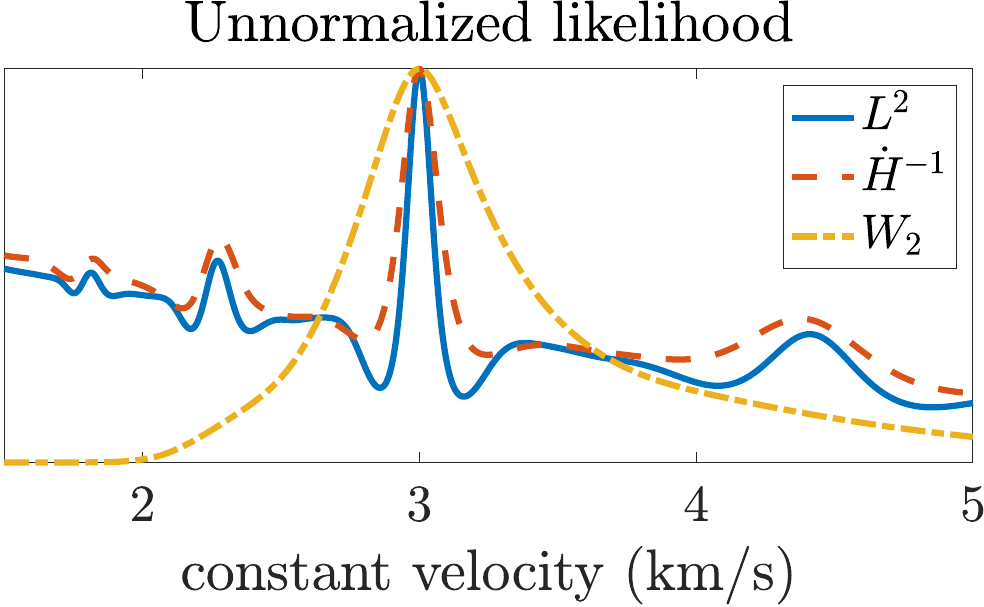}
\includegraphics[width=0.45\textwidth]{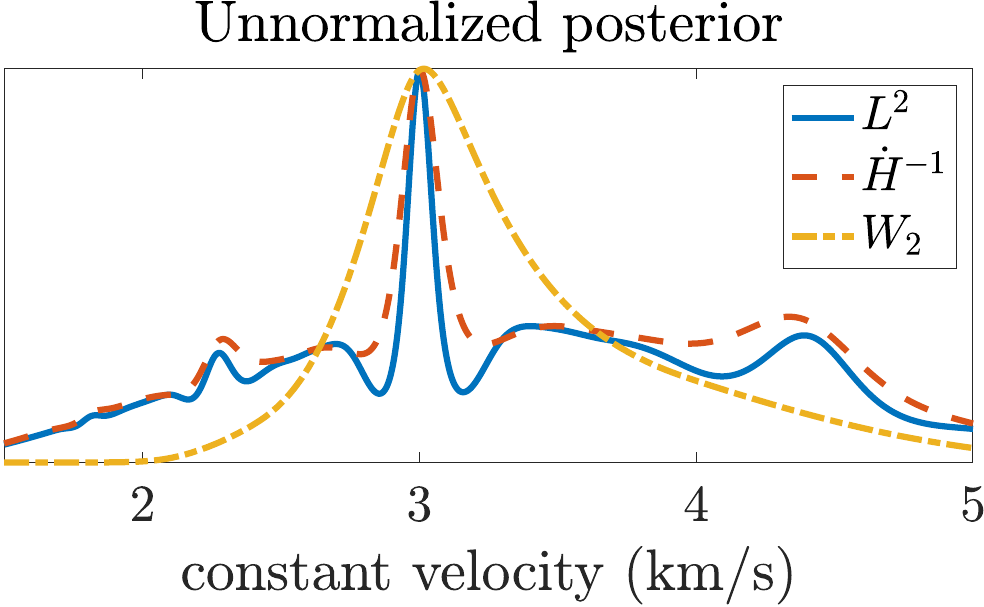}
\caption{(Left) An illustration of the unnormalized likelihood functions for potential $\Phi_{L^2}(u;y),\Phi_{\dot H^{-1}}(u;y),\Phi_{W_2}(u;y)$ by restricting to constant wave speed. (Right) The corresponding unnormalized posterior distributions with the same Gaussian prior.}
\label{fig:simple_posterior}
\end{figure}

\subsection{Continuous Velocity Field}
\label{ssec:num_cts}
In this subsection we consider recovery of a continuous field $m(x) = F(u(x))$. We place a Gaussian prior on $u$, and compare the distributions arising from each choice of loss function. \Cref{fig:truth_1} shows the true velocity field $v = 1/\sqrt{m}$ (0.93 km in depth and 2.31 km in width) along with $N_s = 6$ sources (Ricker wavelet centered at 15 Hz) and $77$ receivers equally distributed in the additional water layer (thickness 0.6 km) at $z=150$ m. The discretization of the wave equations is 30 m in space and 3 ms in time. The total recording time is 1.8 s. The state parameter $u = F^{-1}(m)$ is also shown.

We chose a relatively flat Gaussian prior on $u$ with mean $m_0$ given by a smoothed version of the true field, and covariance function given by the Mat\'ern kernel with parameters $\sigma = 0.7$, $\nu = 3$ and $\ell = 0.05$. This choice of $\sigma$ ensures that the prior on the velocity at each point is approximately uniformly distributed in the range $(v_{\min},v_{\max})$. The prior mean and standard deviation are shown in \cref{fig:prior_1}. The loss functions are all normalized such that $\Phi(m_0) = 1$, then inverse temperature parameter $\beta = 10^{7}$ is fixed -- that is, the potentials $\Phi(u)$ are replaced by $\beta\Phi(u)/\Phi(m_0)$. In practice $\beta$ will be found, for example hierarchically \cite{motamed2019wasserstein}, however we fix it here for a more direct comparison between the distributions; note however that due to the different natures of the likelihood functions, the distributions will not be directly comparable quantitatively.

We denote by $\cL_{\Phi}^{y}$ the Laplace approximation to $\pi_{\Phi}^y$, that is, $\cL_{\Phi}^y = N(m_\Phi,\cC_\Phi)$ where
\begin{align*}
m_\Phi &= \underset{u \in X}{\argmin}\,\Phi(u;y) + \frac{1}{2}\la u,\cC_0^{-1}u\ra,\quad \cC_\Phi^{-1} = \nabla_u^2\Phi(u_\Phi;y) + \cC_0^{-1}.
\end{align*}
In finite dimensions, this approximation arises by performing a second order Taylor approximation of the negative logarithm of the posterior Lebesgue density -- note that when the potential $\Phi$ is quadratic, such as arises when the forward map is linear and the noise is additive Gaussian, this approximation would be exact. Additionally, in the limit of large data or small noise, the posterior and Laplace approximation coincide \cite{schillings2019convergence}. Numerically we calculate the MAP estimates $m_\Phi$ using the L-BFGS algorithm, initialized at a smoothed version of the true field to somewhat avoid multi-modality issues when using the $L^2$ loss. We approximate the covariance $\cC_\Phi$ using the methodology described in \cite[\S 4.2]{isaac2015scalable}.
\begin{remark}
We could instead choose probe the posterior distribution directly using MCMC methods -- though the dimension of the field we are inferring is high, dimension robust methods are available \cite{beskos2017geometric}. However, the statistical performance of such methods is highly dependent on how informative the data is relative to the prior. In the setups we consider the data is highly informative, and so generating a large number of uncorrelated samples using MCMC would likely be computationally prohibitive. Additionally, in the cases of the $L^2$ and $\dot H^{-1}$ likelihoods the posterior is likely to be highly multimodal, as illustrated in the previous subsection, which could lead to additional issues with exploring the full posterior; this would be less of an issue with the $W_2$ likelihood however.
\end{remark}

We consider both clean observations $y = \G(u)$ and noisy observations $y' = \G(u) + \eta$, where $\eta$ has the form
\[
\eta(t,x) = \left(1+\frac{y(t,x)}{\|y\|_{L^\infty}}\right)\eta_0(t)
\]
and $\eta_0$ is temporal white noise; this form of noise is similar to that considered in \cite{yangletter} and does not lie in $L^2$ in the continuous time limit. The resulting SNR is 14.6 dB.

In \cref{fig:post_clean,fig:post_noisy} we show the results of the inversion with the clean and noisy data respectively for the three potentials $\Phi_{L^2}$, $\Phi_{\dot{H}^{-1}}$ and $\Phi_{W_2}$\footnote{We do not consider the potential $\Phi_M$ in this section as it is very sensitive to the choice of mapping $\P_\sigma$; the potential $\Phi_{W_2}$ is also sensitive to this choice. However, the previous study allows for it to be chosen intuitively.}. For the clean data, both the $L^2$ and $W_2$ likelihoods lead to good reconstruction in the mean, though the standard deviation fields are different with the $W_2$ likelihood leading to higher uncertainty in the regions where the velocity is higher. The $\dot{H}^{-1}$ reconstruction is less accurate in the mean but has higher uncertainty to account for this. For the noisy reconstruction, there is a stark difference for the $L^2$ reconstruction, with a poor quality noisy estimate obtained. The $W_2$ reconstruction is only slightly worse than that with the clean data, and again uncertainty is highest in regions with the highest velocity. The $\dot{H}^{-1}$ reconstruction is essentially unchanged from the clean data, though the average uncertainty has decreased slightly. 

These sensitivities can be quantified using the Wasserstein distance between the Laplace approximations, which has an analytic form. We do not consider the Hellinger distance here, as considered in theory -- the resulting Laplace approximations are close to being singular, and so the distances between the approximations are close to the maximum value of 1, whereas the Wasserstein distance is finite even for singular measures. In \cref{tab:distances} we show $d_{\mathrm{Wass}}(\cL_\Phi^y,\cL_\Phi^{y'})$ for each of the potentials $\Phi$ considered, for the fixed noisy realization of the data $y'$ used in \cref{fig:post_noisy}. This illustrates the balance needed between stability and accuracy: $\Phi_{\dot{H}^{-1}}$ leads to a more stable but less accurate posterior than $\Phi_{W_2}$, and conversely $\Phi_{L^2}$ leads to a similarly accurate but much less stable posterior than $\Phi_{W_2}$.

\begin{figure}
\centering
\includegraphics[width=0.8\textwidth,trim=3cm 0cm 3cm 0cm,clip]{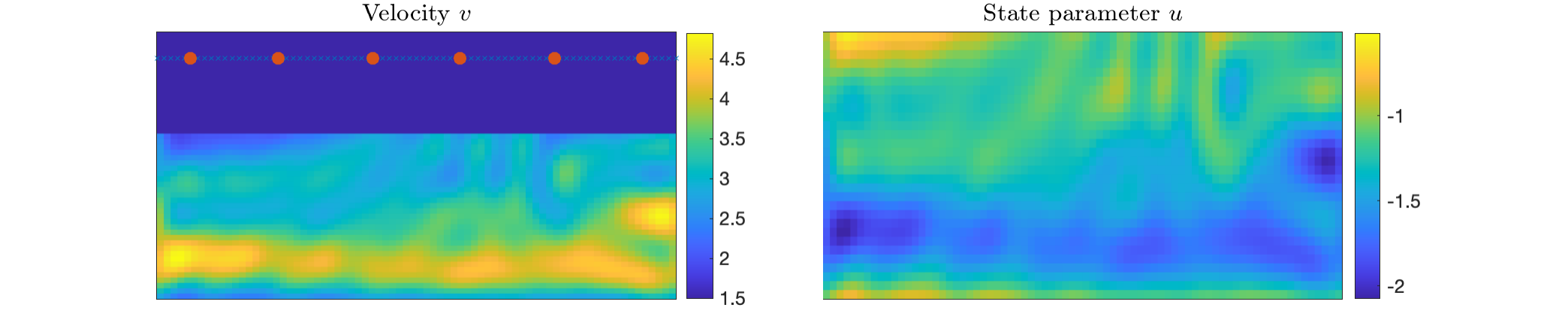}
\caption{(Left) The true continuous velocity field $v$ we aim to infer in \cref{ssec:num_cts}, the location of the six sources $\{s_j\}$ and the set $D_0$ on which the solution is measured at each time. (Right) The state parameter $u = F^{-1}(1/v^2)$, restricted to the domain below the water level.}
\label{fig:truth_1}
\end{figure}

\begin{figure}
\centering
\includegraphics[width=0.8\textwidth,trim=3cm 0cm 3cm 0cm,clip]{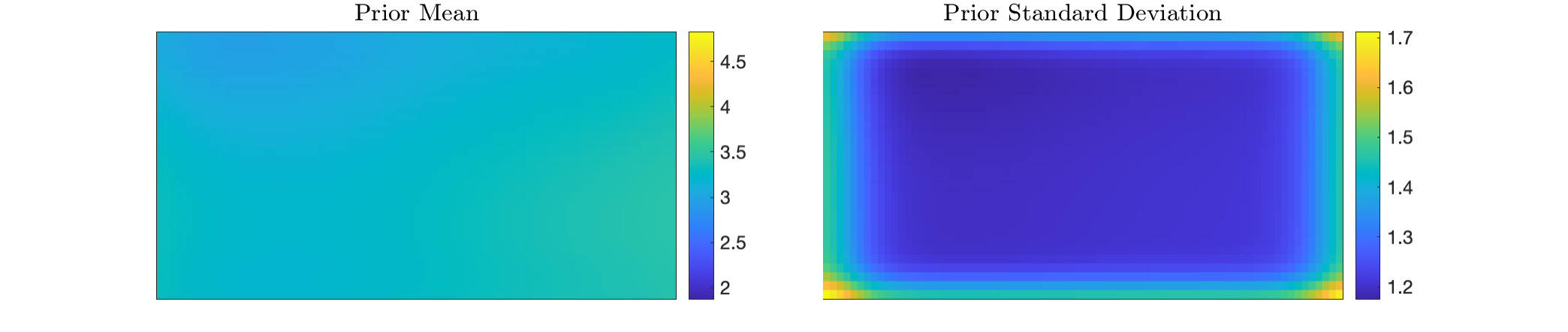}
\caption{The pushforward of the choice of prior mean $m_0(x)$ and standard deviation $\sqrt{c_0(x,x)}$ on $u$ to $v=1/\sqrt{F(u)}$, used in \cref{ssec:num_cts}.}
\label{fig:prior_1}
\end{figure}

\begin{figure}
\centering
\includegraphics[width=0.8\textwidth,trim=3cm 0cm 3cm 0cm,clip]{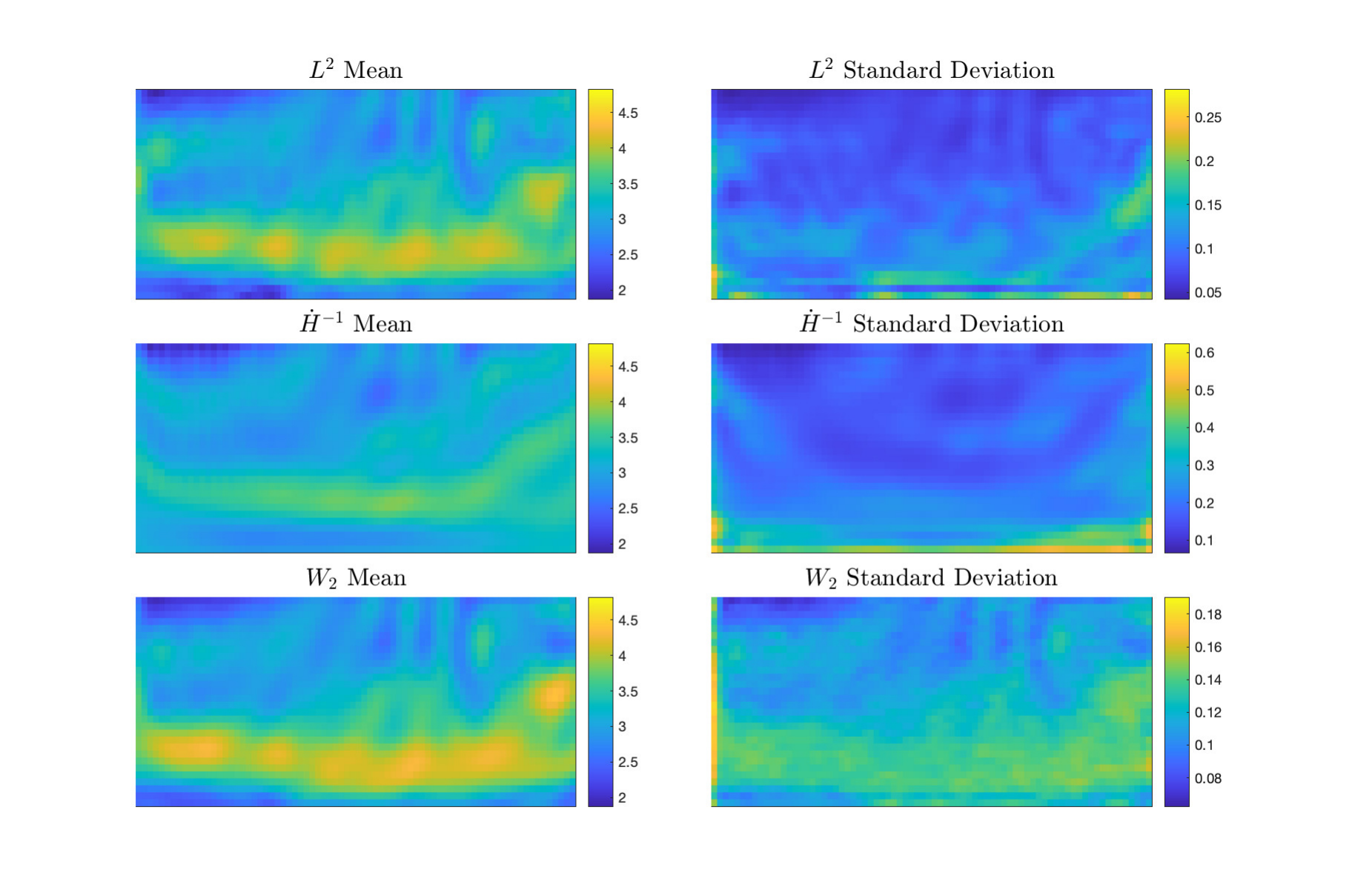}
\caption{The means (left) and standard deviations (right) of the Laplace approximations arising using each of the loss functions $\Phi_{L^2}, \Phi_{\dot H^{-1}}$ and $\Phi_{W_2}$ with clean data $y$ in \cref{ssec:num_cts}.}
\label{fig:post_clean}
\end{figure}

\begin{figure}
\centering
\includegraphics[width=0.8\textwidth,trim=3cm 0cm 3cm 0cm,clip]{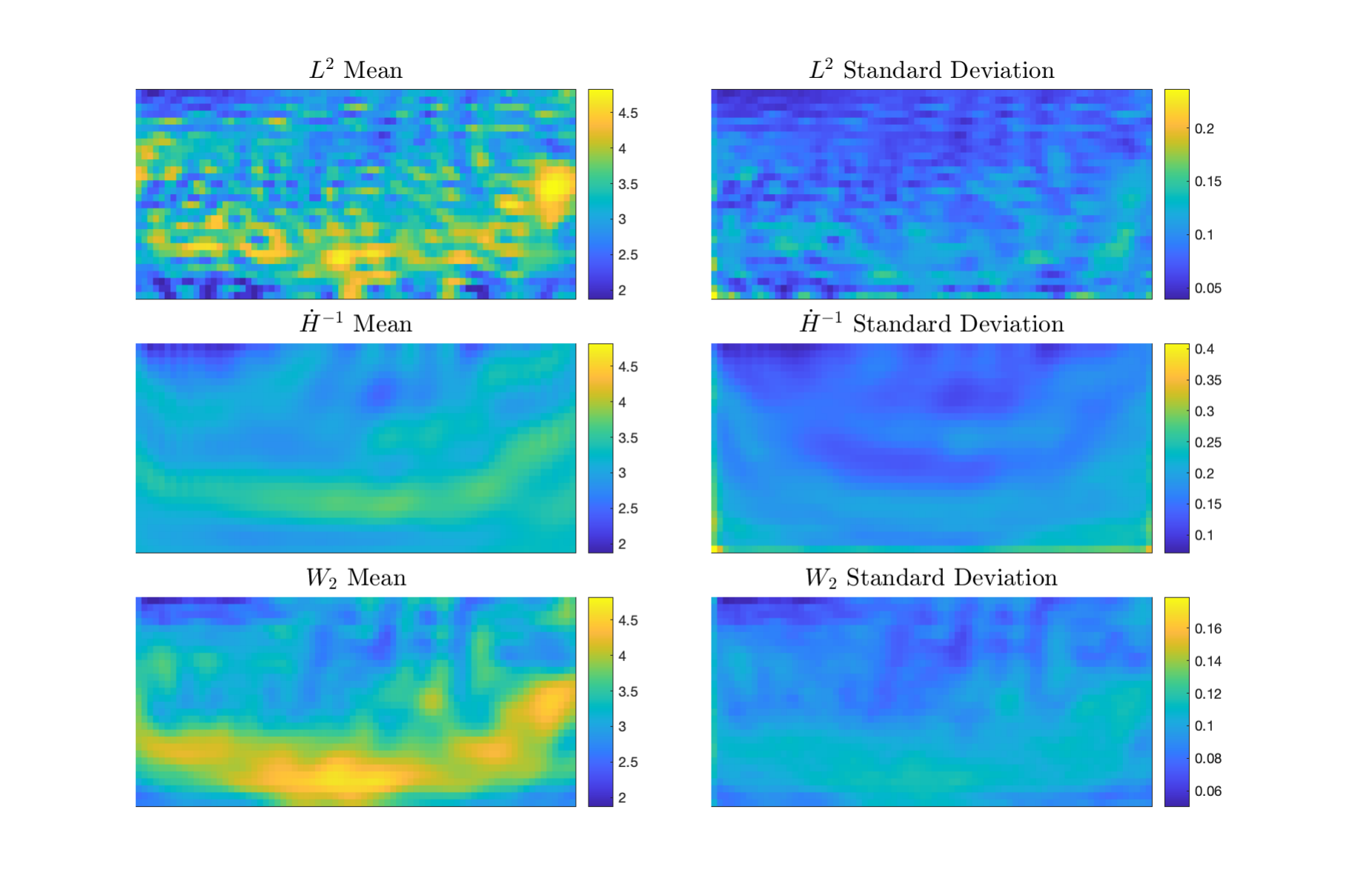}
\caption{The means (left) and standard deviations (right) of the Laplace approximations arising using each of the loss functions $\Phi_{L^2}, \Phi_{\dot H^{-1}}$ and $\Phi_{W_2}$ with noisy data $y'$ in \cref{ssec:num_cts}.}
\label{fig:post_noisy}
\end{figure}

\begin{table}[]
\centering
\caption{Wasserstein distances between the Laplace approximations for the Gibbs posteriors considered in \cref{ssec:num_cts} when perturbed by noise uncorrelated in time.}
\label{tab:distances}
\bgroup
\def\arraystretch{1.5}
\begin{tabular}{l|lll}
 & $\Phi_{L^2}$ & $\Phi_{\dot{H}^{-1}}$ & $\Phi_{W_2}$\\ \hline
$d_{\mathrm{Wass}}(\cL_\Phi^y,\cL_\Phi^{y'})$        & 9.34 & 2.83 & 6.60\\
\end{tabular}
\egroup
\end{table}

\subsection{Salt Model}
\label{ssec:num_salt}
We now consider the recovery of a constant region, corresponding to salt, over a continuous background velocity. The true velocity field (1.95 km in width and 1.11 km in depth) is shown in \cref{fig:truth_salt} along with the source and receiver locations; we increase the number of sources to $N_s = 11$ (Ricker wavelet centered at 10 Hz) in this example while the number of receivers is $65$. They are equally distributed in the additional water layer (thickness 300 m) at $z=150$ m. The discretization of the forward wave equation is 30 m in space and 2 ms in time. The total recording time for the wave propagation is 2 s.
We consider both a Gaussian prior as in the previous subsection and the mixed level set method outlined in \cref{ssec:prior}, choosing the Gaussian field underlying the level set function to have a Mat\'ern covariance prior with parameters $\sigma = 10$, $\nu = 2$ and $\ell = 0.25$, and fixing the continuous background field. The salt velocity is assumed unknown in the mixed level set model, and a relatively flat prior of $N(3,4^2)$ is placed on it; for reference, the true value is 4.79. Numerically the indicator functions in the level set prior \cref{eq:level_set_prior} are replaced with smoothed versions to allow for the computation of gradients.  The potentials are normalized as in the previous subsection, and we choose a lower inverse temperature parameter $\beta = 10^4$ to allow for more influence from the prior. Additionally, we do not corrupt the observations by any noise in this example.

In \cref{fig:post_salt} we show the output when using the potentials $\Phi_{L^2}$, $\Phi_{W_2}$ when using a Gaussian prior, as well as when using $\Phi_{W_2}$ with the mixed level set prior. When a Gaussian prior is used, the $W_2$ reconstruction recovers more of the salt than the $L^2$ reconstruction, though the uncertainty for both is similar. On the other hand, using the level set method, much more of the salt is recovered, and the true value of the salt velocity is recovered almost exactly. The standard deviation illustrates the uncertainty in the boundary location, as well as small areas within the salt where the reconstruction is less accurate.

\begin{figure}
\centering
\includegraphics[width=\textwidth,trim=3cm 0cm 3cm 0cm,clip]{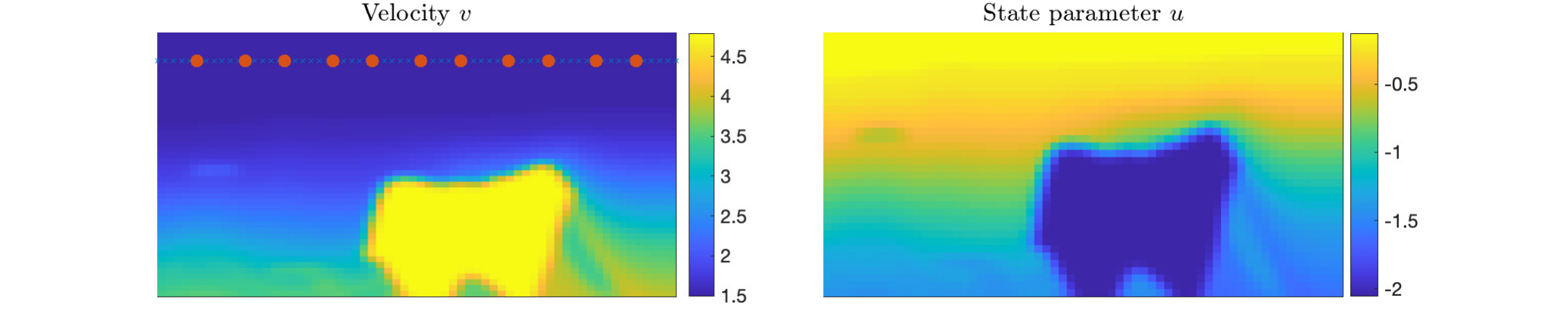}
\caption{(Left) The true salt velocity field $v$ we aim to infer \cref{ssec:num_salt}, the location of the eleven sources $\{s_j\}$ and the set $D_0$ on which the solution is measured at each time. (Right) The state parameter $u = F^{-1}(1/v^2)$, restricted to the domain below the water level.}
\label{fig:truth_salt}
\end{figure}

\begin{figure}
\centering
\includegraphics[width=\textwidth,trim=3cm 0cm 3cm 0cm,clip]{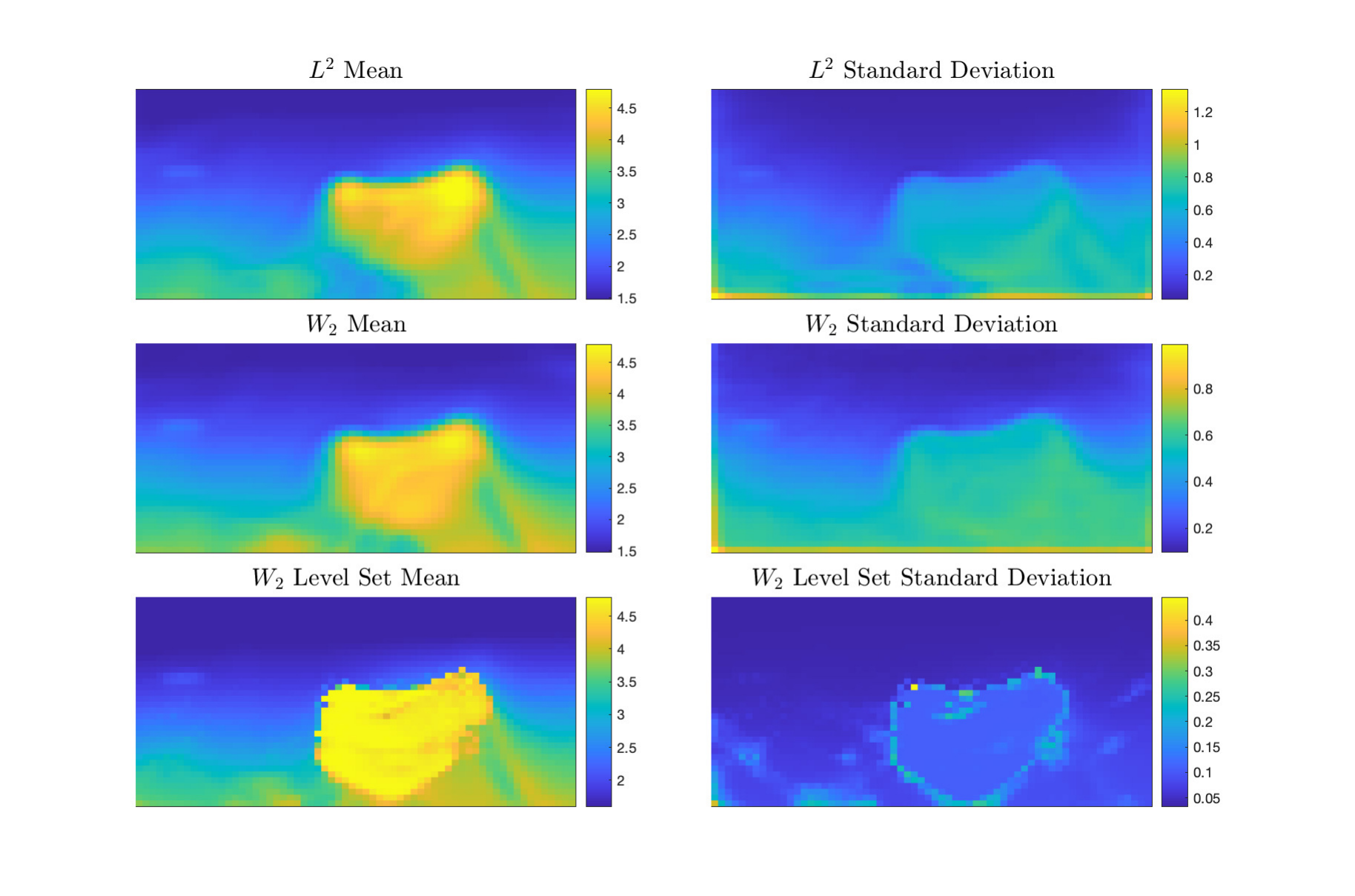}
\caption{The means (left) and standard deviations (right) of the Laplace approximations arising using each of the loss functions $\Phi_{L^2}, \Phi_{\dot H^{-1}}$ and $\Phi_{W_2}$ with clean data $y$ in \cref{ssec:num_salt}.}
\label{fig:post_salt}
\end{figure}


\section{Conclusion} \label{sec:conclusion}
In this paper, we have analyzed the Bayesian properties of four potentials that originate from the objective functions for deterministic full-waveform inversion. We have demonstrated that two of the potentials are equipped with explicit noise models, while the other two can be seen as approximately having state-dependent multiplicative noise that leads to well-defined and stable Gibbs posteriors. As a main component of the paper, our stability results show that posteriors that are based on the Wasserstein loss and the negative Sobolev norm are more stable with respect to high-frequency noise on the data. Numerical inversions under a Bayesian setting also illustrate the robustness and advantages of choosing the Wasserstein-based potentials for seismic inversion.

\section*{Acknowledgement}
This work is partially supported by National Science Foundation grant DMS-1913129, and the U.S. Department of Energy Office of Science, Advanced Scientific Computing Research (ASCR), Scientific Discovery through Advanced Computing (SciDAC) program. We thank Prof. Georg Stadler and Prof. Andrew Stuart for constructive discussions.

\bibliographystyle{plain}
\bibliography{main}

\end{document}